\documentclass[A4]{amsart}

\input xypic
\usepackage{amssymb}
\usepackage{bbm}
\usepackage{amsmath}
\newtheorem{theorem}{Theorem}[section]

\newtheorem{proposition}[theorem]{Proposition}
\newtheorem{lemma}[theorem]{Lemma}

\newtheorem{remark}[theorem]{Remark}

\newcommand{\hlgy}[2]{\ensuremath{H_{#1}(#2)}}
\newcommand{\rhlgy}[2]{\ensuremath{\widetilde{H}_{#1}(#2)}}
\newcommand{\hlgyz}[2]{\ensuremath{H_{#1}(#2;\mathbb{Z}_{p})}}
\newcommand{\rhlgyz}[2]{\ensuremath{\widetilde{H}_{#1}(#2;\mathbb{Z}_{p})}}

\newcounter{bean}

\newenvironment{romanlist}{\begin{list}{\rm ({\roman{bean}})}
      {\usecounter{bean}\setlength{\rightmargin}{\leftmargin}}}
      {\end{list}}

\newcommand{\seqm}[3]{\ensuremath{#1\stackrel{#2}
 {\longrightarrow}#3}}
\newcommand{\seqmm}[5]{\ensuremath{#1\stackrel{#2}
 {\longrightarrow}#3\stackrel{#4}{\longrightarrow}#5}}
\newcommand{\seqmmm}[7]{\ensuremath{#1\stackrel{#2}
 {\longrightarrow}#3\stackrel{#4}{\longrightarrow}#5
  \stackrel{#6}{\longrightarrow}#7}}
\newcommand{\seqmmmm}[9]{\ensuremath{#1\stackrel{#2}
 {\longrightarrow}#3\stackrel{#4}{\longrightarrow}#5
  \stackrel{#6}{\longrightarrow}#7
  \stackrel{#8}{\longrightarrow}#9}}

\newcommand{\floor}[1]{\ensuremath{\left\lfloor #1 \right\rfloor}}

\newcommand{\bracket}[1]{\ensuremath{\left( #1 \right)}}

\newcommand{\abs}[1]{\ensuremath{\left|#1\right|}}

\newcommand{\cwedge}[3]{\displaystyle\bigvee^{#2}_{#1}#3}
\newcommand{\cprod}[3]{\displaystyle\prod^{#2}_{#1}#3}
\newcommand{\ctimes}[3]{\displaystyle\bigotimes^{#2}_{#1}#3}

\newcommand{\csum}[3]{\displaystyle\sum^{#2}_{#1}#3}
\newcommand{\csummulti}[4]{\displaystyle\sum^{#3}
_{\renewcommand{\arraystretch}{0.6}\begin{matrix}\scriptstyle #1 \cr \scriptstyle #2\end{matrix}\renewcommand{\arraystretch}{1.2}}#4}
\newcommand{\tcdots}[0]{\ensuremath{\otimes\cdots\otimes}}

\newcommand{\qqed}{\hfill\square}

\newcommand{\zmodp}{\ensuremath{\mathbb{Z}_{p}}}

\newcommand{\plocal}{\ensuremath{\mathbb{Z}_{(p)}}}

\begin{document}

%%% Title

\title{Modular representations and the homotopy of low rank $p$-local $CW$-complexes}

\author{Piotr Beben} 
\address{\scriptsize{Institute of Mathematics, Academy of Mathematics and Systems Science, 
Chinese Academy of Sciences,
Si Yuan Building,  
Beijing 100190, 
P. R. China}}  
\email{bebenp@unbc.ca}

\author{Jie Wu} 
\address{\scriptsize{Department of Mathematics, National University of Singapore,
Block S17 (SOC1),
10, Lower Kent Ridge Road,
Singapore 119076}}  
\email{matwuj@nus.edu.sg}

%\classification{55P35, 55P99, 20C30.}


\begin{abstract}
Fix an odd prime $p$ and let $X$ be the $p$-localization of a finite suspended $CW$-complex.
Given certain conditions on the reduced mod-$p$ homology $\widetilde H_*(X;\zmodp)$ of $X$,
we use a decomposition of $\Omega\Sigma X$ due to the second author and computations in modular representation theory 
to show there are arbitrarily large integers $i$ such that $\Omega\Sigma^i X$ is a homotopy retract of $\Omega\Sigma X$. 
This implies the stable homotopy groups of $\Sigma X$ are in a certain sense retracts of the unstable homotopy groups,
and by a result of Stanley, one can confirm the Moore conjecture for $\Sigma X$. 
Under additional assumptions on $\widetilde H_*(X;\zmodp)$, we generalize a result of Cohen and Neisendorfer 
to produce a homotopy decomposition of $\Omega\Sigma X$ that has infinitely many finite $H$-spaces as factors.
\end{abstract}

% Leave these items like this, and the journal will fill them in.
%\received{}   % receive date (for example: October 11, 1999)
%\revised{}    % date of revision; omit, if no revision;
                             % if multiple revisions, separate by commas
%\published{Month Day, Year}  % publish date
%\submitted{}  % Name of Journal's Editor, who handled Article 
%\volumeyear{} % Volume Year
%\volumenumber{} % Volume Number 
%\issuenumber{}   % Issue Number
%\startpage{1}     % PageNumber of first page
%\webaddress{}
% If copyright is retained by author, comment this out:
%\owner{International Press}

\maketitle

\section{Introduction}

Finding homotopy decompositions of
loop spaces is of sizeable and broad interest in homotopy theory. 
One of the most fundamental examples is Serre's odd $p$-primary decomposition $\Omega S^{2m}\simeq S^{2m-1}\times\Omega S^{4m-1}$~\cite{Serre}, 
implying the $p$-components of homotopy groups for even dimensional spheres are determined
by those of the odd dimensional spheres. Selick's odd $p$-primary decomposition~\cite{Selick} of the homotopy fiber for the $p$-power map 
\seqm{\Omega^2 S^{2p+1}}{p}{\Omega^2 S^{2p+1}} 
allowed him to find the $p$-exponent of the spehere $S^{3}$, and in combination with an odd $p$-primary decomposition of looped Moore spaces, 
one is lead to the computation for $p$-exponents of higher dimensional spheres (Cohen, Moore, Neisendorfer~\cite{CMN,CMN2}). 

One would like to move beyond these initial successes, and find $p$-primary homotopy decompositions of general $H$-spaces, 
or loop suspensions $\Omega\Sigma X$ for $X$ a $CW$-complex in particular. 
Ideally such a decomposition would be in terms of indecomposable spaces, and one says that the decomposition 
is the \emph{finest} possible. The decomposition problem has been tackled for particular spaces using prior knowledge of self-maps
(see~\cite{MNT2} for example) or universal properties of certain constructed $H$-spaces~\cite{T2}, using techniques related to the Ganea fibration~\cite{Selick7}, and when localized at sufficiently large primes $p$~\cite{MW,Anick1,Anick2,Stelzer1}. 
Most progress to date has involved attempts at \emph{geometrically realizing} some given
coalgebra decompostion of the tensor algebra $\hlgyz{*}{\Omega\Sigma X}\cong T(\rhlgyz{*}{X})$.
In this direction, Cohen and Neisendorfer~\cite{CN1}, as well as Cooke, Harper, and Zabrodsky~\cite{CHZ}, 
extended Serre's decomposition to where $X$ is a certain low rank $CW$-complex, 
and in turn, Selick, Theriault, and Wu~\cite{WS1,WS2,WST} further generalized this to when any simply connected $co$-$H$-space 
is taken in place of $\Sigma X$. 
However, these decompositions were the finest possible only in a functorial sense, and unlike those of Cohen and Neisendorfer,
they were not explicit constructions. 
The computation of the mod-$p$ homology of the factors in the decomposition of Selick, Theriault, and Wu is dependent on open problems in modular representation theory, and as such these decompositions remain rather mysterious.     

Fortunately, there do exist fairly general decompositions that are also explicit. 
In~\cite{Wu2} the second author constructed a homotopy decomposition of $\Omega\Sigma X$ whenever $X$ is a suspension by analysing composites of James-Hopf maps and Samelson products. This turned out to be a partial geometric realization of the Poincar\'e-Birkhoff-Witt coalgebra decomposition
\[ 
T(\rhlgyz{*}{X})\cong\ctimes{i=1}{\infty}{S(L_i(\rhlgyz{*}{X}))},
\]
where $L_i(V)$ is the is the $\zmodp$-submodule of length $i$ Lie brackets in $V$, and $S(V)$ is the free commutative algebra generated by $V$.
Our main interest in this paper is to give more information concerning the factors in this decomposition, which is stated in Theorem~(\ref{Jie}). 

Fix $p$ to be an odd prime throughout. We will be working with $CW$-complexes that have been localized at $p$. A cell structure on a $p$-local space is taken to be in the $p$-local sense. As we will mostly be using reduced mod-$p$ homology, it will be convenient to simply denote it by \rhlgy{*}{X} for any space $X$, without indicating coefficients. Our main result is as follows:

\begin{theorem}
\label{T0}
Let $X$ be the $p$-localization of a suspended $CW$-complex. Set $V=\rhlgy{*}{X}$,
let $M$ denote the sum of the degrees of the generators of $V$, and define the
sequence of integers $b_{i}$ recursively, with $b_{0}=0$ and 
\[b_{i}=(1+\dim V)b_{i-1}+M.\]

Suppose either $V_{odd}=0$ or $V_{even}=0$, and $1<\dim V\leq p$. 
\begin{romanlist}
\item If $\dim V< p-1$, then $\Omega\Sigma^{b_{i}+1}X$ is a homotopy retract of 
$\Omega\Sigma X$ for each $i\geq 1$;
\item if $\dim V=p$, there exist spaces $Y_{i}$ such that $\Omega\Sigma Y_{i}$ is a homotopy retract of 
$\Omega\Sigma X$, and $\rhlgy{*}{Y_{i}}\cong\rhlgy{*}{\Sigma^{b_{i}}X}$ for $i\geq 1$. 
\end{romanlist}
\end{theorem}
\begin{remark}
Notice the case $\dim V=p-1$ is missing in Theorem~(\ref{T0}). Proposition~(\ref{p3}) 
in Section~(\ref{S4}) does not hold when $\dim V=p-1$, while Theorem~(\ref{Jie}) 
would not be applicable even if it did hold.  
\end{remark}

Theorem~(\ref{T0}) will depend on some computational work dealing with 
representations of symmetric groups in Section~(\ref{S1}), 
as well as Wu's decomposition (Theorem~(\ref{Jie})). 
A consequence is that the stable homotopy groups of the 
spaces defined are, in a precise sense, retracts of their regular homotopy groups. This
property leads us to a small application towards the Moore conjecture at the end of this paper.

For the case $\dim V=2$ we prove a stronger version of Theorem~\ref{T0}. We show
the retracts in Theorem~(\ref{T0}) are indeed factors in a decomposition of $\Omega\Sigma X$, 
and that the connectivity of these factors grows at a slower rate. 

\begin{theorem} 
\label{T0a}

Fix $p\geq 5$. Let $X$ be any any suspended $p$-local $CW$-complex with $\dim V=2$, 
and either $V_{odd}=0$ or $V_{even}=0$. 
Let $M$ denoting the sum of the degrees of the two generators of $V$. 

Suppose $0<k_{1}<k_{2}<\cdots$ is any sequence satisfying the following properties:
\begin{enumerate}
\item $2k_{i}+1$ is prime to $p$;
\item $2k_{i}+1$ is not a multiple of $2k_{j}+1$ whenever $i>j$.
\end{enumerate}
Then there exists a decomposition
$$\Omega\Sigma X\simeq \cprod{j}{}{(\Omega\Sigma^{k_{j}M+1}X)}\times(\mbox{Some other space}).$$
\end{theorem}

\begin{remark}
For the case where $X$ has only odd dimensional cells ($V_{even}=0$), 
the authors in~\cite{HaoYuShen} have subsequently strengthened our theorems above,
and given a more elegant proof.  
\end{remark}

Another application is to say a bit more about Cohen's and Neisendorfer's~\cite{CN1}
decompositions involving low rank $CW$-complexes.
Recall this is an extension of Serre's decomposition in precisely the following sense:

\begin{theorem}[Cohen, Neisendorfer]
\label{CN}
Let $X$ be $p$-localization of a suspended $CW$-complex. Set $V=\rhlgy{*}{X}$,
and assume $V_{even}=0$, and $1\leq\dim V<p-1$. 
Then there exists a functorial decomposition
\[\Omega\Sigma X\simeq A(X)\times \Omega Q(X)\]
such that $A(X)$ is a finite $H$-space with
\[\hlgy{*}{A(X)}\cong \Lambda(V)\]
as primitively generated algebras, and
\[\hlgy{*}{\Omega Q(X)}\cong S([L(V),L(V)])=\ctimes{i=2}{\infty}(S(L_{i}(V))),\]
where $L(V)$ is the free graded Lie algebra generated by $V$, and $[L(V),L(V)]$ 
is the sub Lie algebra of $L(V)$ generated by Lie brackets of length greater than one.$\qqed$
\end{theorem}

From the perspective of obtaining finite $H$-spaces, 
we strengthen the above decomposition when $\dim V$ is even as follows:

\begin{theorem}
\label{TCN}
Let $X$ be the $p$-localization of a suspended $CW$-complex. Set $V=\rhlgy{*}{X}$,
and assume $V_{even}=0$, $1<\dim V<p-1$, and $\dim V$ is even. 

Let $M$ denote the sum of the degrees of the generators of $V$, and define the
sequence of integers $b_{i}$ recursively, with $b_{0}=0$ and $b_{i}=(1+\dim V)b_{i-1}+M$.
Then there exists a decomposition
\[\Omega\Sigma X\simeq \cprod{i=0}{\infty}A(\Sigma^{b_{i}} X)\times(\mbox{Some other space}),\]
where (as in Theorem~(\ref{CN})) $A(\Sigma^{b_{i}} X)$ is a finite $H$-space that is a homotopy retract of $\Omega\Sigma^{b_{i}+1} X$,
with
\[\hlgy{*}{A(\Sigma^{b_{i}}X)}\cong \Lambda(\Sigma^{b_{i}}V).\]
\end{theorem}

%+++++++++++++++++++++++++++++++++++++++++++++++++++++++++++++++++++++++++++++++++++++++++++++

\section{Preliminary} 
\label{S1}
In this section, and in Section~(\ref{S2}), $R$ denotes either the field $\zmodp$, or 
the ring of $p$-local integers $\plocal$. 
The symmetric group on $k$ letters is denoted by $S_{k}$, and $R[S_{k}]$ is the group ring over $R$ generated by $S_{k}$. 
The sequence $(a_1,a_2,\cdots, a_k)$ is used to denote the permutation $\sigma\in S_k$ satisfying $\sigma(i)=a_i$.

Let $V$ be any graded $R$-module, and $V^{\otimes k}$ denote the $k$-fold tensor product. 
Consider the (right) action of $R[S_{k}]$ on $V^{\otimes k}$ that is defined by permuting factors in a graded sense. 
In this case, the action of a single element $\sigma\in R[S_{k}]$ on $V^{\otimes k}$ induces a self-map 
\[\seqm{V^{\otimes k}}{\sigma}{V^{\otimes k}},\] 
which we also denote by $\sigma$ for the sake of convenience. 
More generally, when $j_1+j_2+\cdots+j_l=k$, we have the natural pairing
\[
\seqm{R[S_{j_1}]\otimes R[S_{j_2}]\otimes\cdots\otimes R[S_{j_l}]}{}{R[S_{k}]},
\] 
which defines an action of $R[S_{j_1}]\otimes R[S_{j_2}]\otimes\cdots\otimes R[S_{j_l}]$ on $V^{\otimes k}$. 

We will label by $\hat{s}_{k}$ and $\bar{s}_{k}$ the elements in $R[S_{k}]$ defined as the sums
\[\hat{s}_{k} = \csum{\sigma\in S_{k}}{}{sgn(\sigma)\sigma},\]
\[\bar{s}_{k} = \csum{\sigma\in S_{k}}{}{\sigma}.\]
One sees that the action of $\hat{s}_{k}$ and $\bar{s}_{k}$ on $V^{\otimes k}$ sends
a tensor in $V^{\otimes k}$ to a linear combination of all permutations of that tensor. 
For any $x_{1}\tcdots x_{k}\in V^{\otimes k}$ and $\sigma\in S_{k}$, 
one can easily verify the following equalities: 
\begin{lemma}
\label{LUseful}
If each $x_{i}$ has even degree, then
\[
\hat{s}_{k}(x_{\sigma(1)}\tcdots x_{\sigma(k)})=sgn(\sigma)\hat{s}_{k}(x_{1}\tcdots x_{k}).
\] 
Therefore $\hat{s}_{k}(x_{\sigma(1)}\tcdots x_{\sigma(k)})=0$ whenever $x_{\sigma(i)}$ 
and is a multiple of $x_{\sigma(j)}$ for some $i\neq j$. 

Similarly, if each $x_{i}$ has odd degree, then
\[
\bar{s}_{k}(x_{\sigma(1)}\tcdots x_{\sigma(k)})=sgn(\sigma)\bar{s}_{k}(x_{1}\tcdots x_{k}),
\]
and $\bar{s}_{k}(x_{\sigma(1)}\tcdots x_{\sigma(k)})=0$ 
if $x_{\sigma(i)}$ is a multiple of $x_{\sigma(j)}$ for some $i\neq j$.~$\qqed$ 
\end{lemma}

Next, we define the \emph{Dynkin-Specht-Wever} elements $\beta_{k}\in\Sigma_{k}$ inductively, with
$\beta_{2}=1-(2,1)$ and
\[\beta_{k}=(1\otimes\beta_{k-1})(1-(2,3,...,k,1)).\] 
The action of $\beta_{k}$ on $V^{\otimes k}$ is given by sending a tensor $x_{1}\tcdots x_{k}\in V^{\otimes k}$
to the commutator 
\[
[x_1,[x_2,...,[x_{k-1},x_{k}]]\cdots]]\in V^{\otimes k}.
\] 
That is, 
\[
\beta_{2}(x\otimes y) = x\otimes y - (-1)^{\abs{x}\abs{y}}y\otimes x,
\] 
and 
\begin{align}
\beta_{k}(x_{k}\tcdots x_{1}) & = x_{k}\otimes\beta_{k-1}(x_{k-1}\tcdots x_{1})\label{EDSW}\\
&\quad-(-1)^{\abs{x_k}\abs{x_{k-1}\tcdots x_{1}}}\beta_{k-1}(x_{k-1}\tcdots x_{1})\otimes x_{k}\notag.
\end{align}

In the following theorem we define certain integers $c_{n,\ell}$ and $d_{n,\ell}$, 
which will be referred to throughout this paper.  
 
\begin{theorem}
\label{T1}
There exist an integer $c_{n,\ell}\geq 0$ such that for any graded free $R$-module $V$
with $\dim{V}=\ell>1$ and $V_{odd}=0$, and each $x\in V^{\otimes (n\ell+1)}$, we have
\[(\hat{s}_{\ell}^{\otimes n}\otimes 1)\circ\beta_{\ell n+1}\circ(\hat{s}_{\ell}^{\otimes n}\otimes 1)(x),
=\pm c_{n,\ell}(\hat{s}_{\ell}^{\otimes n}\otimes 1)(x).\]
On the other hand, if $V_{even}=0$, there exists an integer $d_{n,\ell}\geq 0$ such that
\[(\bar{s}_{\ell}^{\otimes n}\otimes 1)\circ\beta_{\ell n+1}\circ(\bar{s}_{\ell}^{\otimes n}\otimes 1)(x)
=\pm d_{n,\ell}(\bar{s}_{\ell}^{\otimes n}\otimes 1)(x).\]
These integers are independent of our choice of $V$. 

\end{theorem} 
\begin{proof}

Let $V$ be a free graded $R$-module such that $\ell=\dim{V}>1$ and $V_{odd}=0$, 
and let $g$ denote the self-map 
$(\hat{s}_{\ell}^{\otimes n}\otimes 1)\colon\seqm{V^{\otimes (n\ell+1)}}{}{V^{\otimes (n\ell+1)}}$.
Take a basis $\{\upsilon_{1},...,\upsilon_{\ell}\}$ of $V$, 
and let $y=\upsilon_{1}\tcdots \upsilon_{\ell}\in V^{\otimes\ell}$,
and $z_{i}=y^{\otimes n}\otimes\upsilon_{i}\in V^{\otimes (n\ell+1)}$.
Using Lemma~(\ref{LUseful}), observe $V_{odd}=0$ and $\ell=\dim{V}$ implies  
we can write $g(x)$ as a linear combination of the elements 
$g(z_{i})$ for each $i$ and every $x\in V^{\otimes (n\ell+1)}$, 
and in turn each $g(z_{i})$ is a linear combination of tensors
$\sigma_{1}(y)\tcdots\sigma_{n}(y)\otimes\upsilon_{i}$ ranging over all choices of
$\sigma_{1},...,\sigma_{n}\in S_{n}$.

Let $\gamma\in R[S_{n\ell+1}]$ be any element. Since the factor $\upsilon_{1}$ occurs in $y$, 
\begin{equation}
\label{E1}
g\circ\gamma(y^{\otimes n}\otimes \upsilon_{1})=\pm c_{\gamma}g(y^{\otimes n}\otimes \upsilon_{1})
\end{equation}
for some integer $c_{\gamma}\geq 0$. If we let $y_{j}\in V^{\otimes\ell}$ be the permutation of the 
tensor $y$ such that the first and $j^{th}$ factors are interchanged, it is clear that
$g\circ\gamma(y_{j}^{\otimes n}\otimes \upsilon_{j})=\pm c_{\gamma}g(y_{j}^{\otimes n}\otimes \upsilon_{j})$
for each $j$, as this is the same as replacing $\upsilon_{1}$ with $\upsilon_{j}$ and 
$\upsilon_{j}$ with $\upsilon_{1}$ in equation~(\ref{E1}),
and both $\upsilon_{1}$ and $\upsilon_{j}$ have even degree.
Then since $\hat{s}_{n}(y)=\pm \hat{s}_{n}(y_{j})$ and 
$\hat{s}_{n}(\sigma(y))=sgn(\sigma)\hat{s}_{n}(y)$ for any $\sigma\in S_{n}$, 
$g\circ\gamma(\sigma_{1}(y)\tcdots\sigma_{n}(y)\otimes \upsilon_{j})=
\pm c_{\gamma}g\circ\gamma(\sigma_{1}(y)\tcdots\sigma_{n}(y)\otimes \upsilon_{j})$
for each $j$ and $\sigma_{1},...,\sigma_{n}\in S_{n}$. Since $g(x)$ takes the form of a linear
combination as stated above, thus
\[g\circ\gamma\circ g(x)=\pm c_{\gamma}g(x)\] 
for any $x\in V^{\otimes (n\ell+1)}$. 

Next, let $W$ be any free grade $R$-module such that $\dim W=\dim V$ and $W_{odd}=0$. 
If $\{\omega_{1},...,\omega_{\ell}\}$ is a basis of $W$, there is an 
isomorphism \seqm{V}{\theta}{W} of ungraded $R$-modules defined by sending $\upsilon_{i}$ to $\omega_{i}$. 
Since both $W_{odd}=0$ and $V_{odd}=0$, the isomorphism
\seqm{V^{\otimes (n\ell+1)}}{\theta^{\otimes (n\ell+1)}}{W^{\otimes (n\ell+1)}} of ungraded $R$-modules is equivariant with 
respect to the (graded) action of $R[S_{n\ell+1}]$. Thus $c_{\gamma}$ is independent of $V$. We finish
by setting $\gamma=\beta_{n\ell+1}$, and $c_{n,\ell}=c_{\gamma}$. The proof for the $V_{even}=0$ case is similar.
\end{proof}

The values for several instances of the integers $c_{n,\ell}$ and $d_{n,\ell}$ are given in the following theorem.

\begin{theorem}\mbox{}
\label{T2}
The equality
\[c_{1,\ell}=d_{1,\ell}=(\ell+1)((\ell-1)!)\]
holds for $\ell>1$.
\end{theorem}

\begin{remark}
With some effort one can show $c_{n,2}=d_{n,2}=3^{n}$, though we will not deal with this case in this paper.
Computer culculations also suggest $c_{2,3}=64$, $d_{2,3}=32$; $c_{3,3}=512$, $d_{3,3}=64$; $c_{2,4}=420$, $d_{2,4}=900$. 
Observing the pattern one might conjecture that generally at least one of
$c_{n,\ell}$ or $d_{n,\ell}$ is equal to $(\ell+1)^{n}((\ell-1)!)^{n}$.
 
\end{remark}

%+++++++++++++++++++++++++++++++++++++++++++++++++++++++++++++++++++++++++++++++++++++++++++++

\section{Calculating $c_{1,\ell}$ and $d_{1,\ell}$ for $\ell>1$}
\label{S2}
To show that $c_{1,\ell}=(\ell+1)((\ell-1)!)$, it will suffice to show 
\[
(\hat{s}_{\ell}\otimes 1)\circ\beta_{\ell+1}\circ(\hat{s}_{\ell}\otimes 1)(x)
=\pm (\ell+1)((\ell-1)!)(\hat{s}_{\ell}\otimes 1)(x)
\] 
for any particular choice of graded free $R$-module $V$ satisfying $\ell=\dim{V}>1$ and 
$V_{odd}=0$, and any particular choice of $x\in V$ satisfying $(\hat{s}_{\ell}\otimes 1)(x)\neq 0$. 
We actually work in a slightly more general context. We will show that this equation holds for 
any graded $R$-module $V$ as long as $x$ is of the form $x_{k-1}\tcdots x_{1}\otimes x_{i}$ 
for some $1\leq i\leq k-1$, and each $x_{i}$ is a homogeneous element in $V$ of even degree. 
This is done in Proposition~(\ref{p1}), together with an analogous calculation for $d_{1,\ell}$. 

The following elements in $R[S_{k}$] will be of use: 
\[\hat{t}_{j,k} = \csum{\sigma\in T_{j,k}}{}{sgn(\sigma)\sigma},\]
\[\bar{t}_{j,k} = \csum{\sigma\in T_{j,k}}{}{\sigma},\]
where $T_{j,k}\subseteq S_{k}$ consists of the $k$ elements switching the position of the $j^{th}$ letter
\[
T_{j,k}=\{(j,1,\cdots,j-1,j+1,\cdots,k),(1,j,\cdots,j-1,j+1,\cdots,k),\cdots,(1,\cdots,j-1,j+1,\cdots,k,j)\}.
\] 
It is easy to see these elements satisfy the equations
\begin{equation}
\label{EUseful1}
\hat{s}_{k} = \hat{t}_{1,k}\circ(1\otimes \hat{s}_{k-1}) = \hat{t}_{k,k}\circ(\hat{s}_{k-1}\otimes 1),
\end{equation}
and
\begin{equation}
\label{EUseful2}
\bar{s}_{k} = \bar{t}_{1,k}\circ(1\otimes \bar{s}_{k-1}) = \bar{t}_{k,k}\circ(\bar{s}_{k-1}\otimes 1).
\end{equation}
We begin with a few lemmas.

\begin{lemma}
\label{l1}
Let $k>2$, and $V$ be any graded $R$-module. Then for every $y\in V^{\otimes k}$ 
\[\hat{s}_{k}\circ\beta_{k}(y)=0,\] 
and
\[\bar{s}_{k}\circ\beta_{k}(y)=0.\]
\end{lemma}
\begin{proof}

It is sufficient to show the lemma holds for homogeneous elements in $V^{\otimes k}$. 
We proceed by induction. By inspection we have $\hat{s}_{3}\circ\beta_{3}(x) = 0$ for any $x\in V^{\otimes 3}$. 
Let us assume $\hat{s}_{k-1}\circ\beta_{k-1}(x) = 0$ for every $x\in V^{\otimes k-1}$. 
Take any homogeneous element $x_{k}\tcdots x_{1}\in V^{\otimes k}$ and let $c=\abs{x_{k-1}\tcdots x_{1}}$. 
Using equations~(\ref{EDSW}) and~(\ref{EUseful1})
\begin{align*}
&\hat{s}_{k}\circ\beta_{k}(x_{k}\tcdots x_{1})\\
&= \hat{s}_{k}(x_{k}\otimes\beta_{k-1}(x_{k-1}\tcdots x_{1}))
-(-1)^{c\abs{x_{k}}}\hat{s}_{k}(\beta_{k-1}(x_{k-1}\tcdots x_{1})\otimes x_{k})\\
&= \hat t_{1,k}(1\otimes\hat{s}_{k-1})(x_{k}\otimes\beta_{k-1}(x_{k-1}\tcdots x_{1}))\\
&\quad-(-1)^{c\abs{x_{k}}}\hat t_{k,k}(\hat{s}_{k-1}\otimes 1)(\beta_{k-1}(x_{k-1}\tcdots x_{1})\otimes x_{k})\\ 
&= \hat t_{1,k}(x_{k}\otimes(\hat{s}_{k-1}\circ\beta_{k-1})(x_{k-1}\tcdots x_{1}))\\
&\quad-(-1)^{c\abs{x_{k}}}\hat t_{k,k}((\hat{s}_{k-1}\circ\beta_{k-1})(x_{k-1}\tcdots x_{1})\otimes x_{k})\\ 
&= \hat t_{1,k}(0)-(-1)^{c\abs{x_{k}}}\hat t_{k,k}(0)\\
&=0.
\end{align*}
\noindent This completes the induction step. The proof for $\bar{s}_{k}$ is identical.
\end{proof}

\begin{lemma}
\label{l1a}
Let $k\geq 3$, and $V$ be any graded $R$-module. Then for every $y\in V^{\otimes k}$ and each 
$3\leq j\leq k $, 
\[\hat{s}_{k}\circ(1^{\otimes (k-j)}\otimes\beta_{j})(y)=0,\]
\noindent and
\[\bar{s}_{k}\circ(1^{\otimes (k-j)}\otimes\beta_{j})(y)=0.\]

\end{lemma}

\begin{proof}
It suffices to show the statement holds for homogeneous elements in $V^{\otimes k}$.
We proceed by induction. 
For our inductive assumption, assume the lemma holds for $3\leq m<k$. 
That is, for each $m,j$ such that $3\leq j\leq m<k$, 
and every $x\in V^{\otimes m}$, assume
\[\hat{s}_{m}\circ(1^{\otimes (m-j)}\otimes\beta_{j})(x)=0.\] 
\noindent The base case $m=3$ follows by inspection. 
We must show it also holds for $m=k$ and each $j$ such that $3\leq j\leq k$. 
To do this we use a second induction: for some fixed $j$ assume 
\[\hat{s}_{k}\circ(1^{\otimes (k-i)}\otimes\beta_{i})(y)=0\] 
\noindent holds for every $y\in V^{\otimes k}$ and each $i$ such that $j<i\leq k$. 
The base case $j=k-1$ follows by Lemma~(\ref{l1}). 

Notice by using equation~(\ref{EUseful1}) and our first inductive assumption we have
\begin{align*}
&\hat{s}_{k}(x_{k}\tcdots x_{j+2}\otimes\beta_{j}(x_{j}\tcdots x_{1})\otimes x_{j+1})\\ 
& = \hat t_{k,k}(\hat{s}_{k-1}(x_{k}\tcdots x_{j+2}\otimes\beta_{j}(x_{j}\tcdots x_{1}))\otimes x_{j+1})\\
& = \hat t_{k,k}(0)\\
& = 0.
\end{align*}
\noindent Therefore using (in order) our second inductive assumption, equation~(\ref{EDSW}), and the above equality,
we have  
\begin{align*}
0 & =\hat{s}_{k}\circ(1^{\otimes(k-j-1)}\otimes\beta_{j+1})(x_{k}\tcdots x_{1})\\
& = \hat{s}_{k}(x_{k}\tcdots x_{j+2}\otimes\beta_{j+1}(x_{j+1}\tcdots x_{1}))\\
& = \hat{s}_{k}(x_{k}\tcdots x_{j+2}\otimes x_{j+1}\otimes\beta_{j}(x_{j}\tcdots x_{1}))\\
&\quad-(-1)^{c\abs{x_{j+1}}}\hat{s}_{k}(x_{k}\tcdots x_{j+2}\otimes\beta_{j}(x_{j}\tcdots x_{1})\otimes x_{j+1})\\
& = \hat{s}_{k}(x_{k}\tcdots x_{j+2}\otimes x_{j+1}\otimes\beta_{j}(x_{j}\tcdots x_{1}))\\
& =\hat{s}_{k}\circ(1^{\otimes(k-j)}\otimes\beta_{j})(x_{k}\tcdots x_{1}),
\end{align*}
\noindent
where $c=\abs{x_{j}\tcdots x_{1}}$. 
This completes the second induction, and therefore the main induction step. 
The proof for $\bar{s}_{k}$ is similar.
\end{proof}

\begin{lemma}
\label{l2}
Let $k>2$, and $V$ be any graded $R$-module.
Take any homogeneous element $y=x_{k}\tcdots x_{1}\in V^{\otimes k}$. If $\abs{x_{i}}$ is even for each $i$, then
\[(\hat{s}_{k-1}\otimes 1)\circ\beta_{k}(y) = 
(\hat{s}_{k-1}\otimes 1)\circ(1^{\otimes (k-3)}\otimes ((1,2,3)-(1,3,2)-2(2,3,1)))(y),\]
\noindent and if $\abs{x_{i}}$ is odd for each $i$,
\[(\bar{s}_{k-1}\otimes 1)\circ\beta_{k}(y) = 
(\bar{s}_{k-1}\otimes 1)\circ(1^{\otimes (k-3)}\otimes ((1,2,3)+(1,3,2)-2(2,3,1)))(y).\]
\end{lemma}

\begin{proof}

Assume $\abs{x_{i}}$ is even for each $i$. Using Lemma~(\ref{l1a}) 
\[\hat{s}_{k-1}(x_{k}\tcdots x_{j+1}\otimes\beta_{j-1}(x_{j-1}\tcdots x_{1}))=0\]
\noindent for each $3<j\leq k$, so
\begin{align*}
&(\hat{s}_{k-1}\otimes 1)(x_{k}\tcdots x_{j+1}\otimes\beta_{j}(x_{j}\tcdots x_{1}))\\
&= (\hat{s}_{k-1}\otimes 1)(x_{k}\tcdots x_{j+1}\otimes x_{j}\otimes\beta_{j-1}(x_{j-1}\tcdots x_{1}))\\
&\quad-(\hat{s}_{k-1}(x_{k}\tcdots x_{j+1}\otimes\beta_{j-1}(x_{j-1}\tcdots x_{1})))\otimes x_{j}\\
&= (\hat{s}_{k-1}\otimes 1)(x_{k}\tcdots x_{j+1}\otimes x_{j}\otimes\beta_{j-1}(x_{j-1}\tcdots x_{1})).
\end{align*}
\noindent Then by induction
\[(\hat{s}_{k-1}\otimes 1)\circ\beta_{k}(x_{k}\tcdots x_{1})=
(\hat{s}_{k-1}\otimes 1)(x_{k}\tcdots x_{j+1}\otimes\beta_{j}(x_{j}\tcdots x_{1}))\]
for each $3\leq j\leq k$. In particular, when $j=3$, the fact that 
\[\beta_{3}(x_{3}\otimes x_{2}\otimes x_{1}) = ((1,2,3)-(1,3,2)-(2,3,1)+(3,2,1))(x_{3}\otimes x_{2}\otimes x_{1})\]
\noindent implies
\begin{align*}
&(\hat{s}_{k-1}\otimes 1)\circ\beta_{k}(x_{k}\tcdots x_{1})\\
&=(\hat{s}_{k-1}\otimes 1)\circ(1^{\otimes (k-3)}\otimes ((1,2,3)-(1,3,2)-(2,3,1)+(3,2,1)))(x_{k}\tcdots x_{1}).
\end{align*}
\noindent But Lemma~(\ref{LUseful}) implies
\begin{align*}
&(\hat{s}_{k-1}\otimes 1)\circ(1^{\otimes (k-3)}\otimes (3,2,1))(x_{k}\tcdots x_{1})\\
&=-(\hat{s}_{k-1}\otimes 1)\circ(1^{\otimes (k-3)}\otimes (2,3,1))(x_{k}\tcdots x_{1}).
\end{align*}
\noindent Therefore
\begin{align*}
&(\hat{s}_{k-1}\otimes 1)\circ\beta_{k}(x_{k}\tcdots x_{1})\\ 
&= (\hat{s}_{k-1}\otimes 1)\circ(1^{\otimes (k-3)}\otimes ((1,2,3)-(1,3,2)-2(2,3,1))(x_{k}\tcdots x_{1}).
\end{align*}
\noindent The proof of the second case, when the degrees of each $x_{i}$ are odd, is similar.

\end{proof}

As an immediate consequence of Lemma~(\ref{l2}), we obtain the following lemma.
\begin{lemma}
\label{l3}
Take any homogeneous element $x_{k}\tcdots x_{1}\in V^{\otimes k}$. If $\abs{x_{i}}$ is even for each $i$,
\begin{align*}
(\hat{s}_{k-1}\otimes 1)\circ\beta_{k}(x_{k}\tcdots x_{1}) & =
\begin{cases}
(\hat{s}_{k-1}\otimes 1)(x_{k}\tcdots x_{1}),&\mbox{ if }x_{1}=x_{i}\mbox{ for some }i>3\\
0,&\mbox{ if }x_{1}=x_{2}\\
3(\hat{s}_{k-1}\otimes 1)(x_{k}\tcdots x_{1}),&\mbox{ if }x_{1}=x_{3}\\
\end{cases}
\end{align*}
\noindent
Similarly, if $\abs{x_{i}}$ is odd for each $i$,
\begin{align*}
(\bar{s}_{k-1}\otimes 1)\circ\beta_{k}(x_{k}\tcdots x_{1}) & =
\begin{cases}
(\bar{s}_{k-1}\otimes 1)(x_{k}\tcdots x_{1}),&\mbox{ if }x_{1}=x_{i}\mbox{ for some }i>3\\
0,&\mbox{ if }x_{1}=x_{2}\\
3(\bar{s}_{k-1}\otimes 1)(x_{k}\tcdots x_{1}),&\mbox{ if }x_{1}=x_{3}\\
\end{cases}
\end{align*}

\end{lemma}

\begin{proof}

Suppose each $x_{i}$ is of even degree, and $x_{1}=x_{3}$. By Lemma~(\ref{l2})
\begin{align*}
&(\hat{s}_{k-1}\otimes 1)\circ\beta_{k}(x_{k}\tcdots x_{1}\otimes x_{2}\otimes x_{1})\\
&=(\hat{s}_{k-1}\otimes 1)(1^{\otimes (k-3)}\otimes ((1,2,3)-(1,3,2)-2(2,3,1))(x_{k}\tcdots x_{1}\otimes x_{2}\otimes x_{1})
\end{align*}
\noindent which in turn is equal to
\begin{align*}
&(\hat{s}_{k-1}\otimes 1)(x_{k}\tcdots x_{1}\otimes x_{2}\otimes x_{1})-\\
&(\hat{s}_{k-1}\otimes 1)(x_{k}\tcdots x_{1}\otimes x_{1}\otimes x_{2})-\\
&2(\hat{s}_{k-1}\otimes 1)(x_{k}\tcdots x_{2}\otimes x_{1}\otimes x_{1}).
\end{align*}
\noindent Notice that 
\[(\hat{s}_{k-1}\otimes 1)(x_{k}\tcdots x_{1}\otimes x_{1}\otimes x_{2})=0,\] 
\noindent and since (by Lemma~(\ref{LUseful}))
$\hat{s}_{k-1}(x_{k}\tcdots x_{2}\otimes x_{1})=-\hat{s}_{k-1}(x_{k}\tcdots x_{1}\otimes x_{2})$, we have
\[(\hat{s}_{k-1}\otimes 1)(x_{k}\tcdots x_{2}\otimes x_{1}\otimes x_{1})=
-(\hat{s}_{k-1}\otimes 1)(x_{k}\tcdots x_{1}\otimes x_{2}\otimes x_{1}).\] 
\noindent Therefore we have
\[(\hat{s}_{k-1}\otimes 1)\circ\beta_{k}(x_{k}\tcdots x_{1}\otimes x_{2}\otimes x_{1})
=3(\hat{s}_{k-1}\otimes 1)(x_{k}\tcdots x_{1}\otimes x_{2}\otimes x_{1}).\]
\noindent The proof for the rest of the cases is similar. 

\end{proof}

Theorem~(\ref{T2}) follows from the following proposition.

\begin{proposition}
\label{p1}
Take any homogeneous element $y=x_{k-1}\tcdots x_{1}\in V^{\otimes k-1}$, for $V$ a graded $R$-module. 
If $\abs{x_{i}}$ is even for each $i$, then for each $1\leq j\leq k-1$
\[(\hat{s}_{k-1}\otimes 1)\circ\beta_{k}(\hat{s}_{k-1}\otimes 1)(y\otimes x_{j}) = 
k((k-2)!)(\hat{s}_{k-1}\otimes 1)(y\otimes x_{j}),\]
\noindent and if $\abs{x_{i}}$ is odd for each $i$,
\[(\bar{s}_{k-1}\otimes 1)\circ\beta_{k}(\bar{s}_{k-1}\otimes 1)(y\otimes x_{j}) = 
k((k-2)!)(\bar{s}_{k-1}\otimes 1)(y\otimes x_{j}).\]

\end{proposition}

\begin{proof}

For each $1\leq j,n\leq k-1$, let $S^{j,n}_{k-1}\subseteq S_{k-1}$ be the subset of all permutations $\sigma\in S_{k-1}$ satisfying $\sigma(j)=n$. Notice that  
\[\abs{S^{j,n}_{k-1}}=(k-2)!,\]
\noindent and for $m\neq n$
\[\abs{S^{j,n}_{k-1}\cap S^{j,m}_{k-1}}=0,\]
\noindent and
\[\abs{S_{k-1}-(S^{j,n}_{k-1}\cup S^{j,m}_{k-1})}=(k-1)!-2(k-2)!.\]

\noindent Let $A_{j}$ denote the set $S_{k-1}-(S^{j,1}_{k-1}\cup S^{j,2}_{k-1})$. For each $1\leq j\leq k-1$ we can write $\hat{s}_{k}\in S_{k}$ as
\begin{align*}
\hat{s}_{k} &= \csum{\sigma\in S_{k}}{}{sgn(\sigma)\sigma}\\
&= \csum{\sigma\in A_{j}}{}{sgn(\sigma)\sigma}+\csum{\sigma\in S_{k-1}^{j,1}}{}{sgn(\sigma)\sigma}+\csum{\sigma\in S_{k-1}^{j,2}}{}{sgn(\sigma)\sigma}.
\end{align*}
\noindent Notice that for every $\sigma\in S_{k-1}$ we have 
\[\hat{s}_{k-1}\sigma = sgn(\sigma)\hat{s}_{k-1}.\]
\noindent So for each $\sigma\in S_{k-1}^{j,2}$, 
Lemma~(\ref{l3}) (and~(\ref{LUseful})) implies
\begin{align*}
(\hat{s}_{k-1}\otimes 1)\circ\beta_{k}(\sigma(y)\otimes x_{j}) 
&= 3(\hat{s}_{k-1}\otimes 1)(\sigma(y)\otimes x_{j})\\
&= 3\hat{s}_{k-1}(\sigma(y))\otimes x_{j}\\
&= 3sgn(\sigma)(\hat{s}_{k-1}(y)\otimes x_{j}).
\end{align*}
\noindent Likewise Lemma~(\ref{l3}) implies that for every $\sigma\in A_{j}$, 
\[(\hat{s}_{k-1}\otimes 1)\circ\beta_{k}(\sigma(y)\otimes x_{j}) = sgn(\sigma)(\hat{s}_{k-1}(y)\otimes x_{j})\]
\noindent and for every $\sigma\in S_{k-1}^{j,1}$ we have 
\[(\hat{s}_{k-1}\otimes 1)\circ\beta_{k}(\sigma(y)\otimes x_{j}) = 0.\]
\noindent Hence 
\begin{align*}
(\hat{s}_{k-1}\otimes 1)\circ\beta_{k}\bracket{\csum{\sigma\in S_{k-1}^{j,2}}{}{sgn(\sigma)\sigma(y)}}\otimes x_{j} 
&= \csum{\sigma\in S_{k-1}^{j,2}}{}{(\hat{s}_{k-1}\otimes 1)\circ\beta_{k}(sgn(\sigma)\sigma(y)\otimes x_{j})}\\
&= \csum{\sigma\in S_{k-1}^{j,2}}{}{(3 sgn(\sigma)^2 \hat{s}_{k-1}(y)\otimes x_{j})}\\
&= 3(k-2)!(\hat{s}_{k-1}(y)\otimes x_{j}).
\end{align*}
\noindent Similarly
\begin{align*}
(\hat{s}_{k-1}\otimes 1)\circ\beta_{k}\bracket{\csum{\sigma\in A_{j}}{}{sgn(\sigma)\sigma(y)}}\otimes x_{j} 
&= ((k-1)!-2(k-2)!)(\hat{s}_{k-1}(y)\otimes x_{j}),
\end{align*}
\noindent and
\begin{align*}
(\hat{s}_{k-1}\otimes 1)\circ\beta_{k}\bracket{\csum{\sigma\in S_{k-1}^{j,1}}{}{sgn(\sigma)\sigma(y)}}\otimes x_{j} &= 0.
\end{align*}

Putting these facts together,
\begin{align*}
&(\hat{s}_{k-1}\otimes 1)\circ\beta_{k}(\hat{s}_{k-1}\otimes 1)(y\otimes x_{j})\\ 
&= (\hat{s}_{k-1}\otimes 1)\circ\beta_{k}\bracket{\csum{\sigma\in S_{k}}{}{sgn(\sigma)\sigma(y)}}\otimes x_{j}\\
&=(\hat{s}_{k-1}\otimes 1)\circ\beta_{k}\bracket{\csum{\sigma\in A_{j}}{}{sgn(\sigma)\sigma(y)}+
\csum{\sigma\in S_{k-1}^{j,1}}{}{sgn(\sigma)\sigma(y)}+\csum{\sigma\in S_{k-1}^{j,2}}{}{sgn(\sigma)\sigma(y)}}\otimes x_{j}\\
&=((k-1)!-2(k-2)!)(\hat{s}_{k-1}(y)\otimes x_{j})+0+3(k-2)!(\hat{s}_{k-1}(y)\otimes x_{j})\\
&=k(k-2)!(\hat{s}_{k-1}\otimes 1)(y\otimes x_{j}).
\end{align*}
\noindent The proof for second part of the proposition is similar.
\end{proof}

%+++++++++++++++++++++++++++++++++++++++++++++++++++++++++++++++++++++++++++++++++++++++++++++

\section{Calculating $c_{n,2}$ and $d_{n,2}$ for $n\geq 1$}
\label{S3}
Like in the previous section we will work in the more general context of graded $R$-modules,
and calculate $c_{n,2}$ (and $d_{n,2}$) by proving the equality 
$(\hat{s}_{2}^{\otimes n}\otimes 1)\circ\beta_{2n+1}\circ(\hat{s}_{2}^{\otimes n}\otimes 1)(x)
=\pm (3^{n})(\hat{s}_{2}^{\otimes n}\otimes 1)(x)$ 
holds for certain homogeneous tensors $x$ whose factors are of even degree (or odd degree for the calculation of $d_{n,2}$). 
We begin working our way towards a proof of this starting with a few technical lemmas.

\begin{lemma}
\label{l4}
Let $V$ be any graded $R$-module, and $x_{1},x_{2}\in V$ any homogeneous elements. Let
$\sigma_{1},...,\sigma_{k}\in S_{2}$ be any $k>1$ choices of the two elements in $S_{2}=\{(12),(21)\}$,
and take 
$$y=x_{\sigma_{1}(1)}\otimes 
(x_{\sigma_{k}(1)}\otimes x_{\sigma_{k}(2)}\tcdots x_{\sigma_{2}(1)}\otimes x_{\sigma_{2}(2)})
\otimes x_{\sigma_{1}(2)}\in V^{\otimes 2k},$$
and 
$$z=(x_{\sigma_{k}(1)}\otimes x_{\sigma_{k}(2)}\tcdots x_{\sigma_{2}(1)}\otimes x_{\sigma_{2}(2)})
\otimes x_{\sigma_{1}(2)}\otimes x_{\sigma_{1}(1)}\in V^{\otimes 2k}.$$
\begin{romanlist}
\item Suppose $\abs{x_{1}}$ and $\abs{x_{2}}$ are both odd, and let $n\geq 0$ be the number of 
times $\sigma_{i}=\sigma_{1}$ for $i>1$. Then
$$\bar{s}_{2}^{\otimes k}\circ(1\otimes\beta_{2k-1})(y)
=(-1)^{k-1}(-2)^{n}(\bar{s}_{2}(x_{\sigma_{1}(1)}\otimes x_{\sigma_{1}(2)}))^{\otimes k}.$$
Furthermore, 
$$\bar{s}_{2}^{\otimes k}\circ(\beta_{2k-1}\otimes 1)(z)
=-(\bar{s}_{2}^{\otimes k}\circ(1\otimes\beta_{2k-1})(y)).$$

\item Suppose $\abs{x_{1}}$ and $\abs{x_{2}}$ are both even. If $\sigma_{2i}=\sigma_{1}$ for 
some $i$, then  
$$\hat{s}_{2}^{\otimes k}\circ(1\otimes\beta_{2k-1})(y)=0.$$
Otherwise, if $m\geq 0$ is the number of times $\sigma_{2i+1}=\sigma_{1}$ for $i>0$,
$$\hat{s}_{2}^{\otimes k}\circ(1\otimes\beta_{2k-1})(y)
=(-2)^{m}(3^{\floor{\frac{k}{2}}})(\hat{s}_{2}(x_{\sigma_{1}(1)}\otimes x_{\sigma_{1}(2)}))^{\otimes k}.$$
Furthermore,
$$\hat{s}_{2}^{\otimes k}\circ(\beta_{2k-1}\otimes 1)(z)
=(-1)^{k}(\hat{s}_{2}^{\otimes k}\circ(1\otimes\beta_{2k-1})(y)).$$
\end{romanlist}
\end{lemma}

\begin{proof}[Proof of part (i)]
With our choice of $y\in V^{\otimes 2k}$ defined in the statement of the lemma, 
it will be convenient to let $y^{\prime}\in V^{\otimes 2k-3}$ denote 
$(x_{\sigma_{k-1}(1)} \tcdots x_{\sigma_{2}(2)}\otimes x_{\sigma_{1}(2)})$ -
that is, the tensor of the last $2k-3$ factors of $y$. 
Also, let $n^{\prime}$ be the number of choices of $i$ such that $1<i<k$ and 
$\sigma_{i}=\sigma_{1}$, and $n$ be the number of choices of $i$ such that $1<i\leq k$ and 
$\sigma_{i}=\sigma_{1}$. 

Part (i) holds for $k=2$ by inspection. Assume it holds for some $k-1\geq 2$. In
particular, our inductive assumptions are
\begin{equation}
\label{l4e0a2}
(\bar{s}_{2}^{\otimes k-1})(\beta_{2k-3}(y^{\prime})\otimes x_{\sigma_{1}(1)})
=-(\bar{s}_{2}^{\otimes k-1})(x_{\sigma_{1}(1)}\otimes\beta_{2k-3}(y^{\prime})),
\end{equation} 
and 
\begin{equation}
\label{l4e0c2}
(\bar{s}_{2}^{\otimes k-1})(x_{\sigma_{1}(1)}\otimes\beta_{2k-3}(y^{\prime}))
=(-1)^{k-2}(-2)^{n^{\prime}}
(\bar{s}_{2}(x_{\sigma_{1}(1)}\otimes x_{\sigma_{1}(2)}))^{\otimes k-1}.
\end{equation}

Noting $x_{1}$ and $x_{2}$ are of odd degree, one has the following equality:
\begin{align}
(1\otimes\beta_{2k-1})(y)
&=x_{\sigma_{1}(1)}\otimes x_{\sigma_{k}(1)}\otimes 
\beta_{2k-2}(x_{\sigma_{k}(2)}\tcdots x_{\sigma_{2}(2)}\otimes x_{\sigma_{1}(2)})\label{l4e12}\\
&\quad-x_{\sigma_{1}(1)}\otimes 
\beta_{2k-2}(x_{\sigma_{k}(2)}\tcdots x_{\sigma_{2}(2)}\otimes x_{\sigma_{1}(2)})
\otimes x_{\sigma_{k}(1)}\notag\\
&=x_{\sigma_{1}(1)}\otimes x_{\sigma_{k}(1)}\otimes x_{\sigma_{k}(2)}\otimes 
\beta_{2k-3}(y^{\prime})\notag\\
&\quad+x_{\sigma_{1}(1)}\otimes x_{\sigma_{k}(1)}\otimes 
\beta_{2k-3}(y^{\prime})
\otimes x_{\sigma_{k}(2)}\notag\\
&\quad-x_{\sigma_{1}(1)}\otimes x_{\sigma_{k}(2)}\otimes
\beta_{2k-3}(y^{\prime})
\otimes x_{\sigma_{k}(1)}\notag\\
&\quad-x_{\sigma_{1}(1)}\otimes 
\beta_{2k-3}(y^{\prime})
\otimes x_{\sigma_{k}(2)}\otimes x_{\sigma_{k}(1)}\notag.
\end{align}
Let us assume $\sigma_{1}\neq\sigma_{k}$. Then $x_{\sigma_{k}(2)}=x_{\sigma_{1}(1)}$,
$x_{\sigma_{k}(1)}=x_{\sigma_{1}(2)}$. We replace $x_{\sigma_{k}(2)}$ with $x_{\sigma_{1}(1)}$,
and $x_{\sigma_{k}(1)}$ with $x_{\sigma_{1}(2)}$ in equation~(\ref{l4e12}).
Since $\bar{s}_{2}(x_{\sigma_{1}(1)}\otimes x_{\sigma_{1}(1)})=0$, and  
using equations~(\ref{l4e12}) and~(\ref{l4e0a2}),  
\begin{align}
&(\bar{s}_{2}^{\otimes k})\circ(1\otimes\beta_{2k-1})(y)\notag\\
&=\bar{s}_{2}(x_{\sigma_{1}(1)}\otimes x_{\sigma_{1}(2)})\otimes 
(\bar{s}_{2}^{\otimes k-1})(x_{\sigma_{1}(1)}\otimes\beta_{2k-3}
(y^{\prime}))\notag\\
&\quad-\bar{s}_{2}(x_{\sigma_{1}(1)}\otimes x_{\sigma_{1}(2)})\otimes 
(\bar{s}_{2}^{\otimes k-1})(x_{\sigma_{1}(1)}\otimes\beta_{2k-3}
(y^{\prime}))\notag\\
&\quad-(\bar{s}_{2}^{\otimes k-1})(x_{\sigma_{1}(1)}\otimes 
\beta_{2k-3}(y^{\prime}))
\otimes \bar{s}_{2}(x_{\sigma_{1}(1)}\otimes x_{\sigma_{1}(2)}).\notag\\
&=-(\bar{s}_{2}^{\otimes k-1})(x_{\sigma_{1}(1)}\otimes 
\beta_{2k-3}(y^{\prime}))
\otimes \bar{s}_{2}(x_{\sigma_{1}(1)}\otimes x_{\sigma_{1}(2)}).\notag
\end{align}
Then by equation~(\ref{l4e0c2}),
\begin{align*}
(\bar{s}_{2}^{\otimes k})\circ(1\otimes\beta_{2k-1})(y)
=-(-1)^{k-2}(-2)^{n^{\prime}}
(\bar{s}_{2}(x_{\sigma_{1}(1)}\otimes x_{\sigma_{1}(2)}))^{\otimes k},
\end{align*}
and since our assumption is that $\sigma_{1}\neq \sigma_{k}$, then $n^{\prime}=n$, and so we are done. 

Next, let us assume $\sigma_{1}=\sigma_{k}$. 
Then $x_{\sigma_{k}(1)}=x_{\sigma_{1}(1)}$, and $x_{\sigma_{k}(2)}=x_{\sigma_{1}(2)}$. 
Thus we replace $x_{\sigma_{k}(1)}$ with $x_{\sigma_{1}(1)}$ and
$x_{\sigma_{k}(2)}$ with $x_{\sigma_{1}(2)}$ in equation~(\ref{l4e12}), and using the fact 
that $\bar{s}_{2}(x_{\sigma_{1}(1)}\otimes x_{\sigma_{1}(1)})=0$, we obtain 
\begin{align}
(\bar{s}_{2}^{\otimes k})\circ(1\otimes\beta_{2k-1})(y)
&=-\bar{s}_{2}(x_{\sigma_{1}(1)}\otimes x_{\sigma_{1}(2)})
\otimes\bar{s}_{2}^{\otimes k-1}(\beta_{2k-3}(y^{\prime})\otimes x_{\sigma_{1}(1)})\label{l4e4}\\
&\quad\quad-\bar{s}_{2}^{\otimes k-1}(x_{\sigma_{1}(1)}\otimes\beta_{2k-3}(y^{\prime}))
\otimes \bar{s}_{2}(x_{\sigma_{1}(2)}\otimes x_{\sigma_{1}(1)})\notag\\
&=\bar{s}_{2}(x_{\sigma_{1}(1)}\otimes x_{\sigma_{1}(2)})
\otimes\bar{s}_{2}^{\otimes k-1}(x_{\sigma_{1}(1)}\otimes\beta_{2k-3}(y^{\prime}))\notag\\
&\quad\quad-\bar{s}_{2}^{\otimes k-1}(x_{\sigma_{1}(1)}\otimes\beta_{2k-3}(y^{\prime}))
\otimes (-\bar{s}_{2}(x_{\sigma_{1}(1)}\otimes x_{\sigma_{1}(2)}))\notag
\end{align}
So by equation~(\ref{l4e0c2}), 
\begin{align*}
(\bar{s}_{2}^{\otimes k})\circ(1\otimes\beta_{2k-1})(y)
=2(-1)^{k-2}(-2)^{n^{\prime}}
(\bar{s}_{2}(x_{\sigma_{1}(1)}\otimes x_{\sigma_{1}(2)}))^{\otimes k}\\
=(-1)^{k-1}(-2)^{n^{\prime}+1}
(\bar{s}_{2}(x_{\sigma_{1}(1)}\otimes x_{\sigma_{1}(2)}))^{\otimes k}.
\end{align*}
Since $\sigma_{1}=\sigma_{k}$, then $n^{\prime}+1=n$, and we are done. 

To complete the induction we have to show 
$$\bar{s}_{2}^{\otimes k}\circ(\beta_{2k-1}\otimes 1)(z)
=-(\bar{s}_{2}^{\otimes k}\circ(1\otimes\beta_{2k-1})(y)),$$ 
where $z$ is as defined in the statement of the lemma. 
Having found $(\bar{s}_{2}^{\otimes k})\circ(1\otimes\beta_{2k-1})(y)$ above, this is the same 
as proving the equality
$$\bar{s}_{2}^{\otimes k}\circ(\beta_{2k-1}\otimes 1)(z)
=(-1)^{k}(-2)^{n}(\bar{s}_{2}(x_{\sigma_{1}(1)}\otimes x_{\sigma_{1}(2)}))^{\otimes k}.$$
Thus we follow a similar argument for $z$ as we did for $y$ above. 
First, since $z=x_{\sigma_{k}(1)}\otimes x_{\sigma_{k}(2)}\otimes y^{\prime}\otimes x_{\sigma_{1}(1)}$, 
\begin{align}
(\beta_{2k-1}\otimes 1)(z)
&=x_{\sigma_{k}(1)}\otimes x_{\sigma_{k}(2)}\otimes 
\beta_{2k-3}(y^{\prime})\otimes x_{\sigma_{1}(1)}\label{l4e12i}\\
&\quad+x_{\sigma_{k}(1)}\otimes\beta_{2k-3}(y^{\prime})
\otimes x_{\sigma_{k}(2)}\otimes x_{\sigma_{1}(1)}\notag\\
&\quad-x_{\sigma_{k}(2)}\otimes\beta_{2k-3}(y^{\prime})
\otimes x_{\sigma_{k}(1)}\otimes x_{\sigma_{1}(1)}\notag\\
&\quad-\beta_{2k-3}(y^{\prime})
\otimes x_{\sigma_{k}(2)}\otimes x_{\sigma_{k}(1)}\otimes x_{\sigma_{1}(1)}\notag.
\end{align}
As before, let us first assume $\sigma_{1}\neq\sigma_{k}$. Then
\begin{align*}
(\bar{s}_{2}^{\otimes k})\circ(\beta_{2k-1}\otimes 1)(y)
&=(-\bar{s}_{2}(x_{\sigma_{1}(1)}\otimes x_{\sigma_{1}(2)}))
\otimes\bar{s}_{2}^{\otimes k-1}(\beta_{2k-3}(y^{\prime})\otimes x_{\sigma_{1}(1)})\\
&\quad-(\bar{s}_{2}^{\otimes k-1})
(x_{\sigma_{1}(1)}\otimes \beta_{2k-3}(y^{\prime}))
\otimes (-\bar{s}_{2}(x_{\sigma_{1}(1)}\otimes x_{\sigma_{1}(2)}))\\
&\quad-(\bar{s}_{2}^{\otimes k-1})(\beta_{2k-3}(y^{\prime})\otimes x_{\sigma_{1}(1)})
\otimes (-\bar{s}_{2}(x_{\sigma_{1}(1)}\otimes x_{\sigma_{1}(2)})).
\end{align*}
So by equations~(\ref{l4e0a2}) and~(\ref{l4e0c2}),
\begin{align*}
(\bar{s}_{2}^{\otimes k})\circ(\beta_{2k-1}\otimes 1)(y)
=(-1)^{k}(-2)^{n^{\prime}}
(\bar{s}_{2}(x_{\sigma_{1}(1)}\otimes x_{\sigma_{1}(2)}))^{\otimes k},
\end{align*}
and since we are assuming $\sigma_{1}\neq \sigma_{k}$, then $n^{\prime}=n$, so we are done. 

Next, assume $\sigma_{1}=\sigma_{k}$. By equation~(\ref{l4e12i}) 
\begin{align*}
(\bar{s}_{2}^{\otimes k})\circ(\beta_{2k-1}\otimes 1)(y)
&=\bar{s}_{2}(x_{\sigma_{1}(1)}\otimes x_{\sigma_{1}(2)})
\otimes\bar{s}_{2}^{\otimes k-1}(\beta_{2k-3}(y^{\prime})\otimes x_{\sigma_{1}(1)})\\
&\quad\quad+\bar{s}_{2}^{\otimes k-1}(x_{\sigma_{1}(1)}\otimes\beta_{2k-3}(y^{\prime}))
\otimes (-\bar{s}_{2}(x_{\sigma_{1}(1)}\otimes x_{\sigma_{1}(2)})),
\end{align*}
and so using equation~(\ref{l4e0c2}) we obtain
\begin{align*}
(\bar{s}_{2}^{\otimes k})\circ(\beta_{2k-1}\otimes 1)(y)
=(-1)^{k}(-2)^{n^{\prime}+1}
(\bar{s}_{2}(x_{\sigma_{1}(1)}\otimes x_{\sigma_{1}(2)}))^{\otimes k}.
\end{align*}
Since $\sigma_{1}=\sigma_{k}$, then $n^{\prime}+1=n$ .
This completes the induction for part (i). 
\end{proof}

\begin{proof}[Proof of part (ii)]

We follow along a similar line as the proof of part (i), with $y\in V^{\otimes 2k}$ and 
$y^{\prime}\in V^{\otimes 2k-3}$ defined as before (but this time with $x_{1}$ and $x_{2}$
having both even degree). Let $m^{\prime}$ be the number of
choices of $i$ such that $1<2i+1<k$ and $\sigma_{2i+1}=\sigma_{1}$, and 
$m$ be the number of choices of $i$ such that $1<2i+1\leq k$ and $\sigma_{2i+1}=\sigma_{1}$.

Part (ii) holds for $k=2$ by inspection. Let us assume part (ii) holds for some $k-1\geq 2$. In
particular, our inductive assumptions are as follows: first, 
\begin{equation}
\label{l4e0a}
(\hat{s}_{2}^{\otimes k-1})(\beta_{2k-3}(y^{\prime})\otimes x_{\sigma_{1}(1)})
=(-1)^{k-1}(\hat{s}_{2}^{\otimes k-1})(x_{\sigma_{1}(1)}\otimes\beta_{2k-3}(y^{\prime})),
\end{equation} 
and whenever $\sigma_{1}=\sigma_{2i}$ for some $i$ such that $2i<k$, we have
\begin{equation}
\label{l4e0b}
(\hat{s}_{2}^{\otimes k-1})(x_{\sigma_{1}(1)}\otimes\beta_{2k-3}(y^{\prime}))=0.
\end{equation}
Otherwise,
\begin{equation}
\label{l4e0c}
(\hat{s}_{2}^{\otimes k-1})(x_{\sigma_{1}(1)}\otimes\beta_{2k-3}(y^{\prime}))
=(-2)^{m^{\prime}}(3^{\floor{\frac{k-1}{2}}})(\hat{s}_{2}
(x_{\sigma_{1}(1)}\otimes x_{\sigma_{1}(2)}))^{\otimes k-1}.
\end{equation}

Next, observe the following equality:
\begin{align}
(1\otimes\beta_{2k-1})(y)
&=x_{\sigma_{1}(1)}\otimes x_{\sigma_{k}(1)}\otimes 
\beta_{2k-2}(x_{\sigma_{k}(2)} \tcdots x_{\sigma_{2}(2)}\otimes x_{\sigma_{1}(2)})\label{l4e1}\\
&\quad-x_{\sigma_{1}(1)}\otimes 
\beta_{2k-2}(x_{\sigma_{k}(2)} \tcdots x_{\sigma_{2}(2)}\otimes x_{\sigma_{1}(2)})
\otimes x_{\sigma_{k}(1)}\notag\\
&=x_{\sigma_{1}(1)}\otimes x_{\sigma_{k}(1)}\otimes x_{\sigma_{k}(2)}\otimes 
\beta_{2k-3}(y^{\prime})\notag\\
&\quad-x_{\sigma_{1}(1)}\otimes x_{\sigma_{k}(1)}\otimes 
\beta_{2k-3}(y^{\prime})
\otimes x_{\sigma_{k}(2)}\notag\\
&\quad-x_{\sigma_{1}(1)}\otimes x_{\sigma_{k}(2)}\otimes
\beta_{2k-3}(y^{\prime})
\otimes x_{\sigma_{k}(1)}\notag\\
&\quad+x_{\sigma_{1}(1)}\otimes 
\beta_{2k-3}(y^{\prime})
\otimes x_{\sigma_{k}(2)}\otimes x_{\sigma_{k}(1)}\notag.
\end{align}
Let us assume $\sigma_{1}\neq\sigma_{k}$. Then $x_{\sigma_{k}(2)}=x_{\sigma_{1}(1)}$,
$x_{\sigma_{k}(1)}=x_{\sigma_{1}(2)}$, and so we replace $x_{\sigma_{k}(2)}$ with $x_{\sigma_{1}(1)}$,
and $x_{\sigma_{k}(1)}$ with $x_{\sigma_{1}(2)}$ in equation~(\ref{l4e1}).
Since $\hat{s}_{2}(x_{\sigma_{1}(1)}\otimes x_{\sigma_{1}(1)})=0$, 
$$\hat{s}_{2}^{\otimes k}(x_{\sigma_{1}(1)}\otimes x_{\sigma_{1}(1)}\otimes
\beta_{2k-3}(y^{\prime})
\otimes x_{\sigma_{k}(1)})=0.$$
Therefore by equations~(\ref{l4e1}) and~(\ref{l4e0a}) we have 
\begin{align}
&(\hat{s}_{2}^{\otimes k})\circ(1\otimes\beta_{2k-1})(y)\label{l4e2}\\
&=\hat{s}_{2}(x_{\sigma_{1}(1)}\otimes x_{\sigma_{1}(2)})\otimes 
(\hat{s}_{2}^{\otimes k-1})(x_{\sigma_{1}(1)}\otimes\beta_{2k-3}
(y^{\prime}))\notag\\
&\quad+\hat{s}_{2}(x_{\sigma_{1}(1)}\otimes x_{\sigma_{1}(2)})\otimes 
(\hat{s}_{2}^{\otimes k-1})(\beta_{2k-3}(y^{\prime})\otimes x_{\sigma_{1}(1)})\notag\\
&\quad\hat{s}_{2}^{\otimes k-1})(x_{\sigma_{1}(1)}\otimes\beta_{2k-3}(y^{\prime})
\otimes (-\hat{s}_{2}(x_{\sigma_{1}(1)}\otimes x_{\sigma_{1}(2)})).\notag\\
&=(1-(-1)^{k-1})\hat{s}_{2}(x_{\sigma_{1}(1)}\otimes x_{\sigma_{1}(2)})\otimes 
(\hat{s}_{2}^{\otimes k-1})(x_{\sigma_{1}(1)}\otimes\beta_{2k-3}
(y^{\prime}))\notag\\
&\quad+(\hat{s}_{2}^{\otimes k-1})(x_{\sigma_{1}(1)}\otimes 
\beta_{2k-3}(y^{\prime}))
\otimes \hat{s}_{2}(x_{\sigma_{1}(1)}\otimes x_{\sigma_{1}(2)}).\notag
\end{align}
Now suppose $\sigma_{2i}=\sigma_{1}$ for some $2i<k$. Then 
$(\hat{s}_{2}^{\otimes k-1})(x_{\sigma_{1}(1)}\otimes\beta_{2k-3}(y^{\prime}))=0$ 
by our inductive assumption, and so 
$$(\hat{s}_{2}^{\otimes k})\circ(1\otimes\beta_{2k-1})(y)=0$$ 
by equation~(\ref{l4e2}). On the other hand, suppose $\sigma_{2i}\neq\sigma_{1}$ 
for all $i$. Recall $m^{\prime}$ is the number of
choices of $i$ such that $1<2i+1<k$ and $\sigma_{2i+1}=\sigma_{1}$, and 
$m$ is the number of choices of $i$ such that $1<2i+1\leq k$ and $\sigma_{2i+1}=\sigma_{1}$.
By equation~(\ref{l4e0c})
\begin{align*}
&(\hat{s}_{2}^{\otimes k})\circ(1\otimes\beta_{2k-1})(y)\\
&=(1-(-1)^{k-1})(-2)^{m^{\prime}}(3^{\floor{\frac{k-1}{2}}})
(\hat{s}_{2}(x_{\sigma_{1}(1)}\otimes x_{\sigma_{1}(2)}))^{\otimes k}\\
&\quad+(-2)^{m^{\prime}}(3^{\floor{\frac{k-1}{2}}})
(\hat{s}_{2}(x_{\sigma_{1}(1)}\otimes x_{\sigma_{1}(2)}))^{\otimes k}\\
&=(2+(-1)^{k})(-2)^{m^{\prime}}(3^{\floor{\frac{k-1}{2}}})
(\hat{s}_{2}(x_{\sigma_{1}(1)}\otimes x_{\sigma_{1}(2)}))^{\otimes k}.
\end{align*}
Since $\sigma_{1}\neq \sigma_{k}$, $m^{\prime}=m$. So when $k$ is odd, $\floor{\frac{k-1}{2}}=\floor{\frac{k}{2}}$,
and $(2+(-1)^{k})(-2)^{m^{\prime}}(3^{\floor{\frac{k-1}{2}}})=(-2)^{m}(3^{\floor{\frac{k}{2}}})$. Likewise, when 
$k$ is even, 
$(2+(-1)^{k})(-2)^{m^{\prime}}(3^{\floor{\frac{k-1}{2}}})=3(-2)^{m}(3^{\floor{\frac{k-1}{2}}})=
(-2)^{m}(3^{\floor{\frac{k}{2}}})$. So in any case
$$(\hat{s}_{2}^{\otimes k})\circ(1\otimes\beta_{2k-1})(y)
=(-2)^{m}(3^{\floor{\frac{k}{2}}})(\hat{s}_{2}(x_{\sigma_{1}(1)}\otimes x_{\sigma_{1}(2)}))^{\otimes k}.$$
This finishes the the case $\sigma_{1}\neq\sigma_{k}$. 

Let us assume $\sigma_{1}=\sigma_{k}$. 
Then we have $x_{\sigma_{k}(1)}=x_{\sigma_{1}(1)}$, and
$x_{\sigma_{k}(2)}=x_{\sigma_{1}(2)}$. So replacing $x_{\sigma_{k}(1)}$ with $x_{\sigma_{1}(1)}$ and
$x_{\sigma_{k}(2)}$ with $x_{\sigma_{1}(2)}$ in equation~(\ref{l4e1}), using equation~(\ref{l4e0a}), 
and the fact that $\hat{s}_{2}(x_{\sigma_{1}(1)}\otimes x_{\sigma_{1}(1)})=0$, one obtains 
\begin{align}
(\hat{s}_{2}^{\otimes k})\circ(1\otimes\beta_{2k-1})(y)
&=-(-1)^{k-1}\hat{s}_{2}(x_{\sigma_{1}(1)}\otimes x_{\sigma_{1}(2)})
\otimes\hat{s}_{2}^{\otimes k-1}(x_{\sigma_{1}(1)}\otimes\beta_{2k-3}(y^{\prime}))\notag\\
&\quad+\hat{s}_{2}^{\otimes k-1}(x_{\sigma_{1}(1)}\otimes\beta_{2k-3}(y^{\prime}))
\otimes (-\hat{s}_{2}(x_{\sigma_{1}(1)}\otimes x_{\sigma_{1}(2)}))\notag.
\end{align}
Suppose $\sigma_{1}=\sigma_{2i}$ for some choice of $i$ such that $2i<k$. Then 
$\hat{s}_{2}^{\otimes k-1}(x_{\sigma_{1}(1)}\otimes\beta_{2k-3}(y^{\prime}))=0$ by
equation~(\ref{l4e0b}), and so 
$$(\hat{s}_{2}^{\otimes k})\circ(1\otimes\beta_{2k-1})(y)=0$$
using equation~(\ref{l4e4}). 

Suppose $\sigma_{1}\neq\sigma_{2i}$ for all choices of $i$ 
such that $2i<k$. Then by equations~(\ref{l4e0c}) and~(\ref{l4e4}), 
\begin{equation}
\label{l4e5}
(\hat{s}_{2}^{\otimes k})\circ(1\otimes\beta_{2k-1})(y)
=((-1)^{k}-1)(-2)^{m^{\prime}}(3^{\floor{\frac{k-1}{2}}})
(\hat{s}_{2}(x_{\sigma_{1}(1)}\otimes x_{\sigma_{1}(2)}))^{\otimes k}.
\end{equation}
Since we are assuming $\sigma_{1}=\sigma_{k}$, by the statement of our lemma one would expect that
$(\hat{s}_{2}^{\otimes k})\circ(1\otimes\beta_{2k-1})(y)=0$
whenever $k$ is even. By equation~(\ref{l4e5}) this is indeed the
case. On the other hand when $k$ is odd, since $\sigma_{1}=\sigma_{k}$ we have
$m=m^{\prime}+1$, and therefore by equation~(\ref{l4e5})
$$(\hat{s}_{2}^{\otimes k})\circ(1\otimes\beta_{2k-1})(y)
=(-2)^{m}(3^{\floor{\frac{k}{2}}})(\hat{s}_{2}(x_{\sigma_{1}(1)}\otimes x_{\sigma_{1}(2)}))^{\otimes k}.$$ 
This finishes the $\sigma_{1}=\sigma_{k}$ case. 

To complete the induction we have to show 
$$\hat{s}_{2}^{\otimes k}\circ(\beta_{2k-1}\otimes 1)(z)
=(-1)^{k}(\hat{s}_{2}^{\otimes k}\circ(1\otimes\beta_{2k-1})(y))$$ 
(where $z$ is defined in the statement of the lemma), as an equality of this form in equation~(\ref{l4e0a}) 
is assumed in our induction step. Having found 
$(\hat{s}_{2}^{\otimes k})\circ(1\otimes\beta_{2k-1})(y)$ above, this is the same as proving the equality
$$\hat{s}_{2}^{\otimes k}\circ(\beta_{2k-1}\otimes 1)(z)
=(-1)^{k}(-2)^{m}(3^{\floor{\frac{k}{2}}})(\hat{s}_{2}(x_{\sigma_{1}(1)}\otimes x_{\sigma_{1}(2)}))^{\otimes k}.$$
Here we have $z=x_{\sigma_{k}(1)}\otimes x_{\sigma_{k}(2)}\otimes y^{\prime}\otimes x_{\sigma_{1}(1)}$, whereas before 
$y=x_{\sigma_{1}(1)}\otimes x_{\sigma_{k}(1)}\otimes x_{\sigma_{k}(2)}\otimes y^{\prime}$. 
None-the-less, an argument similar to the one above shows that this equality is correct.
\end{proof}

The following is a consequence of Lemma~(\ref{l4}).

\begin{lemma}
\label{l5}
Let $V$ be any graded $R$-module, and $x_{1},x_{2}\in V$ any homogeneous elements. Let
$\sigma_{1},...,\sigma_{k}\in S_{2}$ be any $k>1$ choices of the two elements 
in $S_{2}=\{(12),(21)\}$, and let 
$$y=(x_{\sigma_{k}(1)}\otimes x_{\sigma_{k}(2)}\tcdots x_{\sigma_{2}(1)}\otimes x_{\sigma_{2}(2)})
\otimes x_{\sigma_{1}(2)}\in V^{\otimes 2k-1}.$$
\begin{romanlist}
\item Suppose $\abs{x_{1}}$ and $\abs{x_{2}}$ are both odd, and let $n\geq 0$ be the number of 
times $\sigma_{i}=\sigma_{1}$ for $i>1$. Then
$$(\bar{s}_{2}^{\otimes k-1}\otimes 1)\circ\beta_{2k-1}(y)
=(-1)^{k-1}(-2)^{n}(\bar{s}_{2}(x_{\sigma_{1}(1)}\otimes x_{\sigma_{1}(2)}))^{\otimes k-1}
\otimes x_{\sigma_{1}(2)}.$$

\item Suppose $\abs{x_{1}}$ and $\abs{x_{2}}$ are both even. If $\sigma_{2i}=\sigma_{1}$ for 
some $i$, then  
$$(\hat{s}_{2}^{\otimes k-1}\otimes 1)\circ\beta_{2k-1}(y)=0.$$
Otherwise, if $m\geq 0$ is the number of times $\sigma_{2i+1}=\sigma_{1}$ for $i>0$, then
$$(\hat{s}_{2}^{\otimes k-1}\otimes 1)\circ\beta_{2k-1}(y)
=(-1)^{k-1}(-2)^{m}(3^{\floor{\frac{k}{2}}})(\hat{s}_{2}(x_{\sigma_{1}(1)}\otimes x_{\sigma_{1}(2)}))^{\otimes k-1}
\otimes x_{\sigma_{1}(2)}.$$
\end{romanlist}
\end{lemma}

\begin{proof}[Proof of part (i)]
Part (i) holds for $k=2$ by inspection. For our inductive assumption, 
assume part (i) holds for some $k-1>2$. Let $y$ be a choice
of tensor as defined in the statement of the lemma, and for convenience let 
$y^{\prime}=x_{\sigma_{k-1}(1)}\otimes x_{\sigma_{k-1}(2)}\tcdots x_{\sigma_{2}(1)}\otimes x_{\sigma_{2}(2)}
\otimes x_{\sigma_{1}(2)}\in V^{\otimes 2k-3}$. Let $n$ be the number of times $\sigma_{i}=\sigma_{1}$ 
for $1<i\leq k$, and $n^{\prime}$ be the number of times $\sigma_{i}=\sigma_{1}$ for $1<i<k$.
Note the following equality:
\begin{align}
(\beta_{2k-1})(y)
&=x_{\sigma_{k}(1)}\otimes x_{\sigma_{k}(2)}\otimes 
\beta_{2k-3}(y^{\prime})\label{l5e1i}\\
&\quad+x_{\sigma_{k}(1)}\otimes 
\beta_{2k-3}(y^{\prime})
\otimes x_{\sigma_{k}(2)}\notag\\
&\quad-x_{\sigma_{k}(2)}\otimes
\beta_{2k-3}(y^{\prime})
\otimes x_{\sigma_{k}(1)}\notag\\
&\quad-\beta_{2k-3}(y^{\prime})
\otimes x_{\sigma_{k}(2)}\otimes x_{\sigma_{k}(1)}\notag.
\end{align}
Let us assume $\sigma_{k}\neq\sigma_{1}$, so we can replace $x_{\sigma_{k}(2)}$ with $x_{\sigma_{1}(1)}$ and
$x_{\sigma_{k}(1)}$ with $x_{\sigma_{1}(2)}$ in equation~(\ref{l5e1i}). In this case, 
observe that the factor $x_{\sigma_{1}(2)}$ occurs more 
often than the factor $x_{\sigma_{1}(1)}$ in the tensor $x_{\sigma_{k}(1)}\otimes y^{\prime}$, thus
$$\bar{s}_{2}^{\otimes k-1}(x_{\sigma_{k}(1)}\otimes\beta_{2k-3}(y^{\prime}))=0.$$
Also, if $\sigma_{1}\neq\sigma_{2i}$ for all $1<2i<k$, then by our inductive assumption
\begin{align*}
(\bar{s}_{2}^{\otimes k-2}\otimes 1)\circ\beta_{2k-3}(y^{\prime})
&=(-1)^{k-2}(-2)^{n^{\prime}}
(\bar{s}_{2}(x_{\sigma_{1}(1)}\otimes x_{\sigma_{1}(2)}))^{\otimes k-2}
\otimes x_{\sigma_{1}(2)},
\end{align*}
and by Lemma~(\ref{l4}),
\begin{align*}
(\bar{s}_{2}^{\otimes k-1})(\beta_{2k-3}(y^{\prime})\otimes x_{\sigma_{1}(1)})
&=-(\bar{s}_{2}^{\otimes k-1})(x_{\sigma_{1}(1)}\otimes\beta_{2k-3}(y^{\prime})).
\end{align*}
Therefore, by equation~(\ref{l5e1i})
\begin{align*}
&(\bar{s}_{2}^{\otimes k-1}\otimes 1)\circ(\beta_{2k-1})(y)\\
&=(-1)^{k-2}(-2)^{n^{\prime}}\bar{s}_{2}(x_{\sigma_{1}(2)}\otimes x_{\sigma_{1}(1)})\otimes
(\bar{s}_{2}(x_{\sigma_{1}(1)}\otimes x_{\sigma_{1}(2)}))^{\otimes k-2}
\otimes x_{\sigma_{1}(2)}\\
&=-(-1)^{k-2}(-2)^{n^{\prime}}
(\bar{s}_{2}(x_{\sigma_{1}(1)}\otimes x_{\sigma_{1}(2)}))^{\otimes k-1}
\otimes x_{\sigma_{1}(2)}.
\end{align*}
Since $\sigma_{k}\neq\sigma_{1}$,
$n=n^{\prime}$, and we are done.  

Next we assume $\sigma_{k}=\sigma_{1}$. Replacing $x_{\sigma_{k}(1)}$ with $x_{\sigma_{1}(1)}$ and
$x_{\sigma_{2}(1)}$ with $x_{\sigma_{1}(2)}$ in equation~(\ref{l5e1i}), 
the factor $x_{\sigma_{1}(2)}$ occurs more often than the factor $x_{\sigma_{1}(1)}$ in the tensors $x_{\sigma_{k}(2)}\otimes y^{\prime}$ and $y^{\prime}\otimes x_{\sigma_{k}(2)}$, implying
$\bar{s}_{2}^{\otimes k-1}(x_{\sigma_{k}(2)}\otimes\beta_{2k-3}(y^{\prime}))=0$, and
$\bar{s}_{2}^{\otimes k-1}(\beta_{2k-3}(y^{\prime})\otimes x_{\sigma_{k}(2)})=0$. Also, by Lemma~(\ref{l4})
$$(\bar{s}_{2}^{\otimes k-1})(x_{\sigma_{1}(1)}\otimes\beta_{2k-3}(y^{\prime}))
=(-1)^{k-2}(-2)^{n^{\prime}}(\bar{s}_{2}(x_{\sigma_{1}(1)}\otimes x_{\sigma_{1}(2)}))^{\otimes k-1}.$$
Following along a similar line as the previous case,  
\begin{align*}
(\bar{s}_{2}^{\otimes k-1}\otimes 1)\circ(\beta_{2k-1})(y)
&=2(-1)^{k-2}(-2)^{n^{\prime}}
(\bar{s}_{2}(x_{\sigma_{1}(1)}\otimes x_{\sigma_{1}(2)}))^{\otimes k-1}
\otimes x_{\sigma_{1}(2)}.
\end{align*}
Because we are assuming $\sigma_{k}\neq\sigma_{1}$, then $n=n^{\prime}+1$, and so 
$2(-1)^{k-2}(-2)^{n^{\prime}}=(-1)^{k-1}(-2)^{n}$. This completes the induction.
\end{proof}

\begin{proof}[Proof of part (ii)]
The structure of the proof is similar to that of part (i).
Here, part (ii) holds for $k=2$ by inspection, and our inductive assumption is that it holds for some $k-1>2$.
Let $y$ and $y^{\prime}$ be as in the proof of part (i), and let $m$ be the number of times $\sigma_{2i+1}=\sigma_{1}$ 
for $0<2i+1\leq k$, and $m^{\prime}$ be the number of times $\sigma_{2i+1}=\sigma_{1}$ for $0<2i+1<k$.
Note the following equality:
\begin{align}
(\beta_{2k-1})(y)
&=x_{\sigma_{k}(1)}\otimes x_{\sigma_{k}(2)}\otimes 
\beta_{2k-3}(y^{\prime})\label{l5e1}\\
&\quad-x_{\sigma_{k}(1)}\otimes 
\beta_{2k-3}(y^{\prime})
\otimes x_{\sigma_{k}(2)}\notag\\
&\quad-x_{\sigma_{k}(2)}\otimes
\beta_{2k-3}(y^{\prime})
\otimes x_{\sigma_{k}(1)}\notag\\
&\quad+\beta_{2k-3}(y^{\prime})
\otimes x_{\sigma_{k}(2)}\otimes x_{\sigma_{k}(1)}\notag.
\end{align}
Let us assume $\sigma_{k}\neq\sigma_{1}$, so we replace $x_{\sigma_{k}(2)}$ with $x_{\sigma_{1}(1)}$ and
$x_{\sigma_{k}(1)}$ with $x_{\sigma_{1}(2)}$ in equation~(\ref{l5e1}). In this case, 
observe that the factor $x_{\sigma_{1}(2)}$ occurs more 
often than the factor $x_{\sigma_{1}(1)}$ in the tensor $x_{\sigma_{k}(1)}\otimes y^{\prime}$, thus
$$\hat{s}_{2}^{\otimes k-1}(x_{\sigma_{k}(1)}\otimes\beta_{2k-3}(y^{\prime}))=0.$$
Also, if $\sigma_{1}\neq\sigma_{2i}$ for all $1<2i<k$, then by our inductive assumption
\begin{align}
(\hat{s}_{2}^{\otimes k-2}\otimes 1)\circ\beta_{2k-3}(y^{\prime})
&=(-1)^{k-2}(-2)^{m^{\prime}}(3^{\floor{\frac{k-1}{2}}})
(\hat{s}_{2}(x_{\sigma_{1}(1)}\otimes x_{\sigma_{1}(2)}))^{\otimes k-2}
\otimes x_{\sigma_{1}(2)}\label{l5e2},
\end{align}
and by Lemma~(\ref{l4}),
\begin{align}
(\hat{s}_{2}^{\otimes k-1})(\beta_{2k-3}(y^{\prime})\otimes x_{\sigma_{1}(1)})
&=(-1)^{k-1}(\hat{s}_{2}^{\otimes k-1})(x_{\sigma_{1}(1)}\otimes\beta_{2k-3}(y^{\prime}))\label{l5e3}\\
&=(-1)^{k-1}(-2)^{m^{\prime}}(3^{\floor{\frac{k-1}{2}}})
(\hat{s}_{2}(x_{\sigma_{1}(1)}\otimes x_{\sigma_{1}(2)}))^{\otimes k-1}\notag.
\end{align}
Therefore
\begin{align*}
(\hat{s}_{2}^{\otimes k-1}\otimes 1)\circ(\beta_{2k-1})(y)
&=(2(-1)^{k-1}-1)(-2)^{m^{\prime}}(3^{\floor{\frac{k-1}{2}}})
(\hat{s}_{2}(x_{\sigma_{1}(1)}\otimes x_{\sigma_{1}(2)}))^{\otimes k-1}
\otimes x_{\sigma_{1}(2)}.
\end{align*}
When $k$ is even, $\floor{\frac{k}{2}}=\floor{\frac{k-1}{2}}+1$ and $(2(-1)^{k-1}-1)=-3=(-1)^{k-1}3$; when $k$ is odd, 
$\floor{\frac{k}{2}}=\floor{\frac{k-1}{2}}$ and $(2(-1)^{k-1}-1)=1=(-1)^{k-1}$. Since $\sigma_{k}\neq\sigma_{1}$,
$m=m^{\prime}$, and we are done. If $\sigma_{1}=\sigma_{2i}$ for some $1<2i<k$, both equations~(\ref{l5e2}) 
and~(\ref{l5e3}) are zero by our inductive assumption and Lemma~(\ref{l4}), and so 
$(\hat{s}_{2}^{\otimes k-1}\otimes 1)\circ(\beta_{2k-1})(y)=0$.  

Finally, let us assume $\sigma_{k}=\sigma_{1}$. Replacing $x_{\sigma_{k}(1)}$ with $x_{\sigma_{1}(1)}$ and
$x_{\sigma_{2}(1)}$ with $x_{\sigma_{1}(2)}$ in equation~(\ref{l5e1}), 
notice that the factor $x_{\sigma_{1}(2)}$ occurs more 
often than the factor $x_{\sigma_{1}(1)}$ in the tensors $x_{\sigma_{k}(2)}\otimes y^{\prime}$ and
$y^{\prime}\otimes x_{\sigma_{k}(2)}$, so
$\hat{s}_{2}^{\otimes k-1}(x_{\sigma_{k}(2)}\otimes\beta_{2k-3}(y^{\prime}))=0$, and
$\hat{s}_{2}^{\otimes k-1}(\beta_{2k-3}(y^{\prime})\otimes x_{\sigma_{k}(2)})=0$. 
The rest being similar as before, we have $(\hat{s}_{2}^{\otimes k-1}\otimes 1)\circ(\beta_{2k-1})(y)=0$
when $\sigma_{1}=\sigma_{2i}$ for some $1<2i<k$, and
\begin{align*}
(\hat{s}_{2}^{\otimes k-1}\otimes 1)\circ(\beta_{2k-1})(y)
&=((-1)^{k-2}-1)(-2)^{m^{\prime}}(3^{\floor{\frac{k-1}{2}}})
(\hat{s}_{2}(x_{\sigma_{1}(1)}\otimes x_{\sigma_{1}(2)}))^{\otimes k-1}
\otimes x_{\sigma_{1}(2)}.
\end{align*}
whenever $\sigma_{1}\neq\sigma_{2i}$ for all $1<2i<k$. When $k$ is even, by the statement of the lemma one would expect
this equation to be zero, as we are assuming $\sigma_{k}\neq\sigma_{1}$. This is in fact the case since 
$((-1)^{k-2}-1)=0$ whenever $k$ is even. On the other hand, when $k$ is odd we have $\floor{\frac{k}{2}}=\floor{\frac{k-1}{2}}$ 
and $((-1)^{k-2}-1)=-2=(-1)^{k-1}(-2)$. Also, because $\sigma_{k}=\sigma_{1}$, then $m=m^{\prime}+1$. This completes
the induction.
\end{proof}

Part (ii) of Theorem~(\ref{T2}) follows from the following proposition. 

\begin{proposition}
\label{pii}
Let $V$ be any graded $R$-module, and $x_{1},x_{2}\in V$ any homogeneous elements. Let
$\sigma_{1},...,\sigma_{k}\in S_{2}$ be any $k>1$ choices of the two elements 
in $S_{2}=\{(12),(21)\}$, and let 
$$y=(x_{\sigma_{k}(1)}\otimes x_{\sigma_{k}(2)}\tcdots x_{\sigma_{2}(1)}\otimes x_{\sigma_{2}(2)})
\otimes x_{\sigma_{1}(2)}\in V^{\otimes 2k-1}.$$
\begin{romanlist}
\item If $\abs{x_{1}}$ and $\abs{x_{2}}$ are both odd, then
$$(\bar{s}_{2}^{\otimes k-1}\otimes 1)\circ\beta_{2k-1}\circ(\bar{s}_{2}^{\otimes k-1}\otimes 1)(y)
=\pm 3^{k-1}(\bar{s}_{2}^{\otimes k-1}\otimes 1)(y).$$
\item If $\abs{x_{1}}$ and $\abs{x_{2}}$ are both even, then  
$$(\hat{s}_{2}^{\otimes k-1}\otimes 1)\circ\beta_{2k-1}\circ(\hat{s}_{2}^{\otimes k-1}\otimes 1)(y)
=\pm 3^{k-1}(\hat{s}_{2}^{\otimes k-1}\otimes 1)(y).$$
\end{romanlist}
\end{proposition}
\begin{proof}

Without loss of generality assume $\sigma_{1}=(21)$ (so $x_{\sigma_{1}(1)}=x_{2}$ and $x_{\sigma_{1}(2)}=x_{1}$), and 
take the tensor 
$$x=(x_{1}\otimes x_{2})^{\otimes k-1}\otimes x_{1}\in V^{\otimes 2k-1}.$$
Notice that $(\bar{s}_{2}^{\otimes k-1}\otimes 1)(y)=(sgn(\sigma_{2})\cdots sgn(\sigma_{k}))(\bar{s}_{2}^{\otimes k-1}\otimes 1)(x)$
when $x_{1}$ and $x_{2}$ have odd degree, and
$(\hat{s}_{2}^{\otimes k-1}\otimes 1)(y)=(sgn(\sigma_{2})\cdots sgn(\sigma_{k}))(\hat{s}_{2}^{\otimes k-1}\otimes 1)(x)$
when $x_{1}$ and $x_{2}$ have even degree.
Since $(sgn(\sigma_{2})\cdots sgn(\sigma_{k}))=\pm 1$, it is sufficient we prove that the proposition holds for $x$ in place of
$y$. In this case, for any collection $\bar{\sigma}_{1},...,\bar{\sigma}_{k-1}\in S_{2}$ we shall write 
$((x_{\bar{\sigma}_{k-1}(1)}\otimes x_{\bar{\sigma}_{k-1}(2)})\tcdots (x_{\bar{\sigma}_{1}(1)}\otimes x_{\bar{\sigma}_{1}(2)}))\otimes x_{1}$
as $(\bar{\sigma}_{k-1}\tcdots\bar{\sigma}_{1}\otimes 1)(x)$ for short. 

Assume $x_{1}$ and $x_{2}$ have odd degree. Then
\begin{align}
&(\bar{s}_{2}^{\otimes k-1}\otimes 1)(x)\label{piie1}\\
&=\csum{\bar{\sigma}_{1},...,\bar{\sigma}_{k-1}\in S_{2}}{}
{(sgn(\bar{\sigma}_{1})\cdots sgn(\bar{\sigma}_{k-1}))(\bar{\sigma}_{k-1}\tcdots\bar{\sigma}_{1}\otimes 1)(x)}\notag\\
&=\csum{n=0}{k-1}{\bracket{\csummulti{\bar{\sigma}_{1},...,\bar{\sigma}_{k-1}\in S_{2},}
{\bar{\sigma}_{i}=(21)\mbox{ \scriptsize for }n\mbox{ \scriptsize choices of }i}{}
{(-1)^{n}(\bar{\sigma}_{k-1}\tcdots\bar{\sigma}_{1}\otimes 1)(x)}}}\notag.
\end{align} 
Also, by part (i) of Lemma~(\ref{l5}), 
\begin{align*}
&(\bar{s}_{2}^{\otimes k-1}\otimes 1)\circ\beta_{2k-1}((\bar{\sigma}_{k-1}\tcdots\bar{\sigma}_{1}\otimes 1)(x))\\
&=(-1)^{k-1}(-2)^{n}(\bar{s}_{2}(x_{1}\otimes x_{2}))^{\otimes k-1}\otimes x_{1}\\
&=(-1)^{k-1}(-2)^{n}(\bar{s}_{2}^{\otimes k-1}\otimes 1)(x)
\end{align*}
whenever $\bar{\sigma}_{i}=(21)$ for exactly $n$ choices of $i$. Since each $\bar{\sigma}_{i}$ can either be $(12)$ or $(21)$,
there are $\binom{k-1}{n}$ choices of $\bar{\sigma}_{1},...,\bar{\sigma}_{k-1}\in S_{2}$ with the property that
$\bar{\sigma}_{i}=(21)$ for exactly $n$ choices of $i$. Therefore by equation~(\ref{piie1})
\begin{align*}
&(\bar{s}_{2}^{\otimes k-1}\otimes 1)\circ\beta_{2k-1}\circ(\bar{s}_{2}^{\otimes k-1}\otimes 1)(x)\\
&=(-1)^{k-1}\csum{n=0}{k-1}{\bracket{\csummulti{\bar{\sigma}_{1},...,\bar{\sigma}_{k-1}\in S_{2},}
{\bar{\sigma}_{i}=(21)\mbox{ \scriptsize for }n\mbox{ \scriptsize choices of }i}{}
{(2^{n})(\bar{s}_{2}^{\otimes k-1}\otimes 1)(x)}}}\\
&=(-1)^{k-1}\bracket{\csum{n=0}{k-1}{\binom{k-1}{n}(2^{n})}}(\bar{s}_{2}^{\otimes k-1}\otimes 1)(x)\\
&=(-1)^{k-1}3^{k-1}(\bar{s}_{2}^{\otimes k-1}\otimes 1)(x),
\end{align*}
where the last equality follows by the binomial formula.

On the other hand, assume $x_{1}$ and $x_{2}$ have even degree. Then 
\begin{align*}
(\hat{s}_{2}^{\otimes k-1}\otimes 1)(x)
&=\csum{m=0}{k-1}{\bracket{\csummulti{\bar{\sigma}_{1},...,\bar{\sigma}_{k-1}\in S_{2},}
{\bar{\sigma}_{i}=(21)\mbox{ \scriptsize for }m\mbox{ \scriptsize choices of }i}{}
{(-1)^{m}(\bar{\sigma}_{k-1}\tcdots\bar{\sigma}_{1}\otimes 1)(x)}}}.
\end{align*} 
By part (ii) of Lemma~(\ref{l5}), if $\bar{\sigma}_{2i}=(21)$ for some $i$, 
$(\hat{s}_{2}^{\otimes k-1}\otimes 1)\circ\beta_{2k-1}((\bar{\sigma}_{k-1}\tcdots\bar{\sigma}_{1}\otimes 1)(x))=0.$
Otherwise if $\bar{\sigma}_{2i+1}=(21)$ for exactly $m$ choices of $i$, then we have
\begin{align*}
&(\hat{s}_{2}^{\otimes k-1}\otimes 1)\circ\beta_{2k-1}((\bar{\sigma}_{k-1}\tcdots\bar{\sigma}_{1}\otimes 1)(x))\\
&=(-1)^{k-1}(-2)^{m}(3^{\floor{\frac{k}{2}}})(\hat{s}_{2}^{\otimes k-1}\otimes 1)(x).
\end{align*}
So because there are $\binom{\floor{\frac{k-1}{2}}}{m}$ choices of $\bar{\sigma}_{1},...,\bar{\sigma}_{k-1}\in S_{2}$ with 
the property that $\bar{\sigma}_{2i+1}=(21)$ for exactly $m$ choices of $i$, and $\bar{\sigma}_{2i}\neq(21)$ for each $i$,
then
\begin{align*}
&(\hat{s}_{2}^{\otimes k-1}\otimes 1)\circ\beta_{2k-1}\circ(\hat{s}_{2}^{\otimes k-1}\otimes 1)(x)\\
&=(-1)^{k-1}(3^{\floor{\frac{k}{2}}})\bracket{\csum{m=0}{\floor{\frac{k-1}{2}}}
{\binom{\floor{\frac{k-1}{2}}}{m}(2^{m})}}(\hat{s}_{2}^{\otimes k-1}\otimes 1)(x)\\
&=(-1)^{k-1}(3^{\floor{\frac{k}{2}}})(3^{\floor{\frac{k-1}{2}}})(\hat{s}_{2}^{\otimes k-1}\otimes 1)(x)\\
&=(-1)^{k-1}(3^{k-1})(\hat{s}_{2}^{\otimes k-1}\otimes 1)(x),
\end{align*}
where again the last equalities follow using the binomial formula.
\end{proof}

%+++++++++++++++++++++++++++++++++++++++++++++++++++++++++++++++++++++++++++++++++++++++++++++

\section{Geometrically realizing Theorem~(\ref{T1})}
\label{S4}

In this section we fix $X$ to be the $p$-localization of a suspended $CW$-complex of finite type. 
Denote the reduced $\zmodp$-homology of $X$ by $V$. 
The $k$-fold self-smash of $X$ is written as $X^{(k)}$, 
is also the $p$-localization of a suspended $CW$-complex of finite type, 
and its reduced mod-$p$ homology given by 
\[
\rhlgy{*}{X^{(k)}}\cong V^{\otimes k}.
\]

Fix an element $\sigma\in\zmodp[S_{k}]$.
The action of $\zmodp[S_{k}]$ on $V^{\otimes k}$ induces the self-map 
\seqm{V^{\otimes k}}{\sigma}{V^{\otimes k}}. 
One constructs a self-map 
\seqm{X^{(k)}}{f_{\sigma}}{X^{(k)}} 
with property that $(f_{\sigma})_{*}$ induces 
\seqm{V^{\otimes k}}{\sigma}{V^{\otimes k}} on mod-$p$ homology as follows.
Write $\sigma=c_1\sigma_1+\cdots c_j\sigma_j$ for $\sigma_i\in S_{k}$ and each
$c_i$ a nonzero element in $\zmodp$.
We have the canonical action of $S_k$ on $X^{(k)}$ given by permuting its factors.
Thus each element $\sigma_i$ defines a map \seqm{X^{(k)}}{f_{\sigma_i}}{X^{(k)}}.
Now define $f_{\sigma}$ as the composite
$$
f_{\sigma}\colon\seqmmm{X^{(k)}}{pinch}{\cwedge{i=1}{j}{X^{(k)}}}{\vee_i (f_{\sigma_i}\circ c_i)}{\cwedge{i=1}{j}{X^{(k)}}}{fold}{X^{(k)}}, 
$$  
where the pinch map is defined via the $co$-$H$-space structure on $X^{(k)}$ 
(since $X$ is a suspension, so is $X^{(k)}$), and \seqm{X^{(k)}}{c_i}{X^{(k)}} is the degree $c_i$ map.

We will denote the mapping telescope
\[
\seqmm{X^{(k)}}{f_{\sigma}}{X^{(k)}}{f_{\sigma}}{\cdots}
\]  
by $T(f_{\sigma})$. 
If $\sigma\in\zmodp[S_{k}]$ happens to be an idempotent, 
$(1-\sigma)\in\zmodp[S_{k}]$ is also an idempotent orthogonal to $\sigma$.
We have isomorphisms 
\[
\rhlgy{*}{T(f_{\sigma})}\simeq Im(\sigma\colon\seqm{V^{\otimes k}}{}{V^{\otimes k}})
\]
\[
\rhlgy{*}{T(f_{1-\sigma})}\simeq Im((1-\sigma)\colon\seqm{V^{\otimes k}}{}{V^{\otimes k}}),
\]
and the canonical inclusions 
\[
\seqm{X^{(k)}}{\iota_1}{T(f_{\sigma})}
\]
\[
\seqm{X^{(k)}}{\iota_2}{T(f_{1-\sigma})} 
\]
induce the respective projections maps on mod-$p$ homology. Then the composite 
\[
\seqmm{X^{(k)}}{pinch}{X^{(k)}\vee X^{(k)}}{\iota_1\vee\iota_2}{T(f_{\sigma})\vee T(f_{1-\sigma})}
\]
induces an isomorphism on mod-$p$ homology. 
Since $X^{(k)}$ is the $p$-localization of a finite type
$CW$-complex, this composite is a homotopy equivalence.
In particular, the inclusion $\iota_1$ is a homotopy retraction, 
and its right homotopy inverse \seqm{T(f_{\sigma})}{\kappa}{X^{(k)}} can be chosen so that the composite 
\seqmm{X^{(k)}}{\iota}{T(f_{\sigma})}{\kappa}{{X^{(k)}}} is homotopic to $f_{\sigma}$.

It is well known that $\beta_{k}\beta_{k}=k\beta_{k}\in\zmodp[S_{k}]$~\cite{Specht,Wever}, 
and as such $\frac{1}{k}{\beta_{k}}\in\zmodp[S_{k}]$ is an idempotent whenever $k$ is prime to $p$. 
By taking mapping telescopes one has $T(f_{\beta_{k}})$ a homotopy retract of $X^{(k)}$ when $k$ is prime to $p$. 
We shall denote $T(f_{\beta_{k}})$ by $L_{k}(X)$. The 
mod-$p$ homology of $L_{k}(X)$ is the image of \seqm{V^{\otimes k}}{\beta_{k}}{V^{\otimes k}}, so  
it is the submodule of length $k$ Lie brackets in $V^{\otimes k}$, and is denoted by $L_{k}(V)$. 
These spaces $L_{k}(X)$ turn out to have some significance, 
as is apparent in the following homotopy decomposition (see~\cite{Wu2}).

\begin{theorem}
\label{Jie}
Let $1<k_{1}<k_{2}<\cdots$ be any sequence satisfying the
following properties:
\begin{enumerate}
\item $k_{i}$ is prime to $p$;
\item $k_{i}$ is not a multiple of $k_{j}$ whenever $i>j$.
\end{enumerate}
Then there exists a homotopy decomposition
\[\Omega\Sigma X\simeq \cprod{j}{}{\Omega\Sigma L_{k_{j}}(X)}\times(\mbox{Some other space}).\]
$\qqed$
\end{theorem}

The $L_{k}(X)$ are therefore of direct interest to homotopy theory of $\Omega\Sigma X$.
It would be ideal to split these into more familiar spaces. 
One could start by searching for splittings of $X^{(k)}$, 
and then use the fact that $L_{k}(X)$ is a homotopy retract of $X^{(k)}$. 
A comprehensive look at the finest possible $2$-primary splittings of 
$X^{(k)}$ for $X$ a $2$-cell complex can be found in~\cite{WS3}. 

The following proposition gives a criteria for the existence of a
certain homotopy retract of $L_{n\ell+1}(X)$.

\begin{proposition}
\label{p3}
Fix $n>0$ and $\ell>1$ such that $n\ell+1$ is prime to $p$, and take the integers 
$c_{n,\ell}$ and $d_{n,\ell}$ in Theorem~(\ref{T1}). 

Suppose $\dim V=\ell>1$ (where $V$ denotes \rhlgy{*}{X}). 
Let $M$ denote the sum of the degrees of the generators of $V$.
If $c_{n,\ell}$ is prime to $p$ and $V_{odd}=0$, or $d_{n,\ell}$ is prime to $p$ and $V_{even}=0$,
then
\begin{romanlist}
\item there exist a space $Y$ that is a homotopy retract of 
$L_{n\ell+1}(X)$, and $\rhlgy{*}{Y}\cong\rhlgy{*}{\Sigma^{nM}X}$; 
\item if $\ell\leq p-1$, then $\Sigma^{nM}X$ is a homotopy retract of 
$L_{n\ell+1}(X)$.
\end{romanlist}
\end{proposition}

\begin{proof}[Proof of part (i)]
Recall the elements $\bar{s}_{\ell},\hat{s}_{\ell}\in \zmodp[S_{\ell}]$ defined in Section~(\ref{S1}). If $V_{even}=0$, 
let $s_{\ell}=\bar{s}_{\ell}$ and assume $c=c_{n,\ell}$ is prime to $p$. Otherwise if 
$V_{odd}=0$, let $s_{\ell}=\hat{s}_{\ell}$ and assume $c=d_{n,\ell}$ is prime to $p$. 
We have self-maps \seqm{X^{(\ell)}}{f_{s_{\ell}}}{X^{(\ell)}} and \seqm{X^{(n\ell+1)}}{f_{\beta_{n\ell+1}}}{X^{(n\ell+1)}} inducing 
\seqm{V^{\otimes\ell}}{s_{\ell}}{V^{\otimes\ell}} and \seqm{V^{\otimes(n\ell+1)}}{\beta_{n\ell+1}}
{V^{\otimes (n\ell+1)}} 
on mod-$p$ homology. Consider the composite
\[g\colon\seqmm{X^{(n\ell+1)}}{f_{s_{\ell}}^{(n)}\wedge\mathbbm{1}}{X^{(n\ell+1)}}{f_{\beta_{n\ell+1}}}{X^{(n\ell+1)}},\]
where $\mathbbm{1}$ is the identity map on $X$, and $f_{s_{\ell}}^{(n)}$ is the $n$-fold self-smash of $f_{s_{\ell}}$.
On mod-$p$ homology $g$ induces 
\[g_{*}\colon\seqmm{V^{\otimes(n\ell+1)}}{s_{\ell}^{\otimes n}\otimes 1}{V^{\otimes (n\ell+1)}}{\beta_{n\ell+1}}{V^{\otimes (n\ell+1)}}.\]
Let $T(g)$ be the telescope of $g$. By Theorem~(\ref{T1}), 
\begin{equation}
\label{T3e1}
(s_{\ell}^{\otimes n}\otimes 1)\circ\beta_{n\ell+1}\circ(s_{\ell}^{\otimes n}\otimes 1)
=c(s_{\ell}^{\otimes n}\otimes 1).
\end{equation}
Thus $g_{*}\circ g_{*}=c(g_{*})$. Since $c$ is prime to $p$, this implies $\rhlgy{*}{T(g)}\cong Im(g_{*})$. Notice 
$Im(g_{*})\subseteq Im(s_{\ell}^{\otimes n}\otimes 1)$, and $Im((s_{\ell}^{\otimes n}\otimes 1)\circ g_{*})
\subseteq Im(g_{*}\circ g_{*})= Im(g_{*})$.
By equation~(\ref{T3e1}),  
$Im((s_{\ell}^{\otimes n}\otimes 1)\circ g_{*})=Im(c(s_{\ell}^{\otimes n}\otimes 1))
=Im(s_{\ell}^{\otimes n}\otimes 1)$. Therefore $Im(g_{*})=Im(s_{\ell}^{\otimes n}\otimes 1)$. Also, 
$Im(s_{\ell}^{\otimes n})$ is 
a submodule of $V^{\otimes n\ell}$ with dimension $1$, whose single generator has degree $nM$. So
$Im(s_{\ell}^{\otimes n}\otimes 1)\cong \Sigma^{nM}V$ as graded $\zmodp$-modules, where $\Sigma^{nM}V$ is 
the $nM$-fold suspension of the graded $\zmodp$-module $V$. Hence 
\[\rhlgy{*}{T(g)}\cong Im(g_{*})\cong \Sigma^{nM}V\cong\rhlgy{*}{\Sigma^{nM}X}.\]

Let us also consider the composite
\[h=(\underline{c}-g)\colon\seqmmm{X^{(n\ell+1)}}{\psi}{X^{(n\ell+1)}\vee X^{(n\ell+1)}}
{\underline{c}\vee-g}{X^{(n\ell+1)}\vee X^{(n\ell+1)}}{\triangledown}{X^{(n\ell+1)}},\]
where $\psi$ is the pinch map, $\underline{c}$ is the degree $c$ map on $X^{(n\ell+1)}$, $-g$ is the composite
of $g$ and the degree $-1$ map on $X^{(n\ell+1)}$, and $\triangledown$ is the fold map.
On mod-$p$ homology we have $h_{*}=c-g_{*}$. Since $c$ is prime to $p$, $Im(g_{*})+Im(h_{*})=V^{\otimes (n\ell+1)}$. 
But $g_{*}\circ g_{*}=c(g_{*})$, so $Im(g_{*})\cap Im(h_{*})=0$.
Therefore $V^{\otimes (n\ell+1)}$ splits as a sum of $\zmodp$-submodules $Im(g_{*})\oplus Im(h_{*})$.  
Notice that 
\[h_{*}\circ h_{*}=(c-g_{*})\circ (c-g_{*})=c^{2}-2c(g_{*})+(g_{*}\circ g_{*})=c^{2}-2c(g_{*})+c(g_{*})=c(h_{*}),\] 
so taking the telescope $T(h)$, we have $\rhlgy{*}{T(h)}\cong Im(h_{*})$. 
Thus we have the following splitting of graded $\zmodp$-modules,
\[\rhlgy{*}{X^{(n\ell+1)}}= V^{\otimes (n\ell+1)}= Im(g_{*})\oplus Im(h_{*})
\cong\rhlgy{*}{T(g)}\oplus \rhlgy{*}{T(h)}.\]
As the inclusions \seqm{X^{(n\ell+1)}}{\iota_{g}}{T(g)} and \seqm{X^{(n\ell+1)}}{\iota_{h}}{T(h)} induce projections 
of $Im(g_{*})$ and $Im(h_{*})$ isomorphically onto 
\rhlgy{*}{T(g)} and \rhlgy{*}{T(h)} in mod-$p$ homology, the map 
\[f\colon\seqmm{X^{(n\ell+1)}}{\psi}{X^{(n\ell+1)}\vee X^{(n\ell+1)}}{\iota_{g}\vee \iota_{h}}{T(g)\vee T(h)}\]  
induces an isomorphism on mod-$p$ homology. Since $X^{(n\ell+1)}$ is the $p$-localization of a finite
type $CW$-complex, $f$ is a homotopy equivalence. 

Let $f^{-1}\colon\seqm{T(g)\vee T(h)}{}{X^{(n\ell+1)}}$ denote the
inverse homotopy equivalence of $f$. Since the composite $f\circ f^{-1}\colon\seqm{X^{(n\ell+1)}}{}{X^{(n\ell+1)}}$ is 
homotopic to the identity, and \seqm{X^{(n\ell+1)}}{\iota_{g}}{T(g)} induces a map of $Im(g_{*})$
isomorphically onto \rhlgy{*}{T(g)}, the composite 
\[\kappa_{g}\colon\seqmm{T(g)}{}{T(g)\vee T(h)}{f^{-1}}{X^{(n\ell+1)}}\] 
maps \rhlgy{*}{T(g)} isomorphically
onto $Im(g_{*})$ in mod-$p$ homology. Also, since $n\ell+1$ is prime to $p$ and 
$\beta_{n\ell+1}\circ\beta_{n\ell+1}=(n\ell+1)\beta_{n\ell+1}$,
$\frac{1}{n\ell+1}\beta_{n\ell+1}\in\zmodp[S_{n\ell+1}]$ is an idempotent. So the inclusion 
\seqm{X^{(n\ell+1)}}{\iota}{T(f_{\beta_{n\ell+1}})=L_{n\ell+1}(X)} is a homotopy retraction, and we 
can take some left homotopy inverse $\kappa$ such that
\seqmm{X^{(n\ell+1)}}{\iota}{L_{n\ell+1}(X)}{\kappa}{X^{(n\ell+1)}}
is homotopic to \seqm{X^{(n\ell+1)}}{f_{\beta_{n\ell+1}}}{X^{(n\ell+1)}}. Now consider the composite
\[\alpha\colon\seqmmmm{T(g)}{\kappa_{g}}{X^{(n\ell+1)}}{\iota}{L_{n\ell+1}(X)}{\kappa}{X^{(n\ell+1)}}{\iota_{g}}{T(g)}.\]
Recall $g_{*}=\beta_{n\ell+1}\circ(s_{\ell}^{\otimes n}\otimes 1)$ by definition. Then on mod-$p$ homology $\kappa_{*}\circ\iota_{*}$ 
sends $Im(g_{*})$ surjectively onto 
$\beta_{n\ell+1}(Im(g_{*}))=Im(\beta_{n\ell+1}\circ\beta_{n\ell+1}\circ(s_{\ell}^{\otimes n}\otimes 1))
=Im((n\ell+1)\beta_{n\ell+1}\circ(s_{\ell}^{\otimes n}\otimes 1))=Im(g_{*}).$
Since $(\iota_{g})_{*}$ projects $Im(g_{*})$ isomorphically onto \rhlgy{*}{T(g)}, 
and $\kappa_{g}$ maps \rhlgy{*}{T(g)} isomorphically onto $Im(g_{*})$, $\alpha_{*}$ is an isomorphism
on mod-$p$ homology. 
Since $T(g)$ is a summand in the above splitting of $X^{(n\ell+1)}$, 
which is the $p$-localization of a finite type $CW$-complex, 
$\alpha$ must be a homotopy equivalence, and so $T(g)$ is a homotopy retract of $L_{n\ell+1}(X)$. 
\end{proof}

\begin{proof}[Proof of part (ii)]
We continue where the proof of part (i) left off to avoid redefining things. 
This time we assume $\ell\leq p-1$. 
Recall $s_{\ell}\in\zmodp[S_{\ell}]$ is either $\bar{s}_{\ell}$, or $\hat{s}_{\ell}$, depending on whether
$V_{even}=0$ or $V_{odd}=0$. In either case it is well known (and not difficult to see) that
\[
s_{\ell}s_{\ell}=\ell!s_{\ell},
\]
so one can take the idempotent $\frac{1}{\ell!}s_{\ell}$ when $\ell\leq p-1$. 
The inclusion \seqm{X^{(\ell)}}{\bar{\iota}}{T(f_{s_{\ell}})} 
is therefore a homotopy retraction, 
and since $\rhlgy{*}{T(f_{s_{\ell}})}\cong Im(s_{\ell})$ is a $1$-dimensional submodule of 
$V^{\otimes n\ell}$ whose generator has degree $M$, 
$T(f_{s_{\ell}})$ is homotopy equivalent to the $M$-sphere $S^{M}$.

Let $\gamma\colon\seqm{X^{(n\ell)}}{\bar{\iota}^{(n)}}{S^{nM}}$ be the $n$-fold smash of $\bar{\iota}$.
On mod-$p$ homology $\gamma$ induces an isomorphism onto $Im(s_{\ell}^{\otimes n})$, 
and so the smash of $\gamma$ and the identity on $X$, 
\[\gamma\wedge\mathbbm{1}\colon\seqm{X^{(n\ell+1)}}{}{\Sigma^{nM}X},\] 
induces an isomorphism onto $Im(s_{\ell}^{\otimes n}\otimes 1)$. Since the section map \seqm{T(g)}{\kappa_{g}}{X^{(n\ell+1)}} 
defined in the proof of part (i) induces an isomorphism onto $Im(s_{\ell}^{\otimes n}\otimes 1)$ on mod-$p$ homology, 
the composite \seqmm{T(g)}{\kappa_{g}}{X^{(n\ell+1)}}{\gamma\wedge\mathbbm{1}}{\Sigma^{nM}X} is an isomorphism
on mod-$p$ homology, so it is a homotopy equivalence. Thus part (ii) follows from part (i).

\end{proof}

If the criteria in Proposition~(\ref{p3}) are satisfied, Proposition~(\ref{p3}) together with 
Theorem~(\ref{Jie}) imply $\Omega\Sigma Y$ is a homotopy retract of $\Omega\Sigma X$. 
Similarly, when $\ell\leq p-1$, $\Omega\Sigma^{nM+1}X$ is a homotopy retract of $\Omega\Sigma X$.
If both $d_{n,\ell}$ and $c_{n,\ell}$ are prime to $p$, one can iterate Proposition~(\ref{p3}).

\begin{proposition}
\label{p4}
Fix $n>0$ and $\ell>1$ such that $n\ell+1$ is prime to $p$, and 
suppose both $c_{n,\ell}$ and $d_{n,\ell}$ are prime to $p$. 

Let $X$ be any suspended $p$-local $CW$-complex with $\dim V=\ell>1$ (where $V$ denotes \rhlgy{*}{X}),
and either $V_{odd}=0$ or $V_{even}=0$. 
Let $M$ denote the sum of the degrees of the generators of $V$, and define the
sequence of integers $b_{i,n}$ recursively with $b_{0,n}=0$, and 
\[b_{i,n}=(n\ell+1)b_{i-1,n}+nM.\]
Then
\begin{romanlist}
\item there exist spaces $Y_{i}$ such that $\Omega\Sigma Y_{i}$ is a homotopy retract of 
$\Omega\Sigma X$, and $\rhlgy{*}{Y_{i}}\cong\rhlgy{*}{\Sigma^{b_{i,n}}X}$ for each $i\geq 1$; 
\item if $\ell\leq p-1$, then $\Omega\Sigma^{b_{i,n}+1}X$ is a homotopy retract of 
$\Omega\Sigma X$ for each $i\geq 1$.
\end{romanlist}
\end{proposition}

\begin{proof}
We will prove part (ii) since part (i) is similar. This is done by induction, with the base
case being $\Omega\Sigma^{nM+1}X$ is a homotopy retract of $\Omega\Sigma X$. This base case holds true
since $\Sigma^{nM}X$ is a homotopy retract of $L_{n\ell+1}(X)$ by Proposition~(\ref{p3}),
and by Theorem~(\ref{Jie}) $\Omega\Sigma L_{n\ell+1}(X)$ is a homotopy retract of $\Omega\Sigma X$ when $n\ell+1$
is prime to $p$.

For our inductive assumption, let us assume $\Omega\Sigma^{b_{i,n}+1}X$ is a homotopy retract of 
$\Omega\Sigma X$ for some $i\geq 1$, and
let $M^{\prime}$ be the sum of the degrees of the generators of \rhlgy{*}{\Sigma^{b_{i,n}}X}. 
Notice that $\dim \Sigma^{b_{i,n}}V=\dim V=\ell$, and since $V_{odd}=0$ or $V_{even}=0$, 
either $(\Sigma^{b_{i,n}}V)_{odd}=0$ or $(\Sigma^{b_{i,n}}V)_{even}=0$. 
Since we have an isomorphism $\rhlgy{*}{\Sigma^{b_{i,n}}X}\cong\Sigma^{b_{i,n}}V$ 
of graded $\zmodp$-modules,
and since we are assuming that both $c_{n,\ell}$ and $d_{n,\ell}$ are prime to $p$,
by Proposition~(\ref{p3}) $\Sigma^{nM^{\prime}}(\Sigma^{b_{i,n}}X)$ is a homotopy retract 
of $L_{n\ell+1}(\Sigma^{b_{i,n}}X)$, 
Also, because $n\ell+1$ is prime to $p$, by Theorem~(\ref{Jie}) 
$\Omega\Sigma L_{n\ell+1}(\Sigma^{b_{i,n}}X)$ is a homotopy retract of $\Omega\Sigma (\Sigma^{b_{i,n}}X)$,
so $\Omega\Sigma^{nM^{\prime}+1}(\Sigma^{b_{i,n}}X)$ is also a homotopy retract of 
$\Omega\Sigma(\Sigma^{b_{i,n}}X)$. Then using our inductive assumption, 
$\Omega\Sigma^{nM^{\prime}+1}(\Sigma^{b_{i,n}}X)$ is a homotopy retract of $\Omega\Sigma X$. 

To check that $M^{\prime}$ has the correct value,  
let $\{\upsilon_{1},...,\upsilon_{\ell}\}$ be a basis for $V$ and $M$ be the sum of
the degrees of the generators in this basis. In this case
\[M^{\prime}=\csum{1\leq i\leq\ell}{}{(b_{i,n}+\abs{\upsilon_{i}})=\ell b_{i,n}+M}.\]
Thus $\Sigma^{nM^{\prime}}(\Sigma^{b_{i,n}}X)=\Sigma^{b_{i+1,n}}X$. This completes
the induction.

\end{proof}

\begin{proof}[Proof of Theorem~(\ref{T0})] 
In Theorem~(\ref{T2}) we found that $c_{1,\ell}=d_{1,\ell}=(\ell+1)((\ell-1)!)$ for $\ell>1$. 
Theorem~(\ref{T0}) now follows as a direct consequence of Proposition~(\ref{p4}). 
\end{proof}

\begin{proof}[Proof of Theorem~(\ref{TCN})]
Let $X$ be any suspended $p$-local $CW$-complex, and let $V$ denote \rhlgy{*}{X}, and $M$ be
the sum of the degrees of the generators in $V$.
Assume $V_{even}=0$, $\ell=\dim V$ is even, and $1<\ell<p-1$. Let $L(V)$ is the free Lie algebra
generated by $V$, and $[L(V),L(V)]$ 
the sub Lie algebra of $L(V)$ generated by Lie brackets of length greater than one.
By the Poincare\text{\'{e}}-Birkhoff-Witt theorem, there is an isomorphism of coalgebras 
\[T(V)\cong \Lambda(V)\otimes S([L(V),L(V)]).\]
This isomorphism is geometrically realized by Cohen's and Neisendorfer's 
decomposition (Theorem~(\ref{CN}))
\[\Omega\Sigma X\simeq A(X)\times \Omega Q(X),\]
with $\hlgy{*}{A(X)}\cong \Lambda(V)$ and $\hlgy{*}{\Omega Q(X)}\cong S([L(V),L(V)])$.
By Theorem~(\ref{Jie}) $\Omega\Sigma L_{\ell+1}(X)$ is a homotopy retract of $\Omega\Sigma X$, 
and the proof of this in~\cite{Wu2} indicates the section map \seqm{\Omega\Sigma L_{\ell+1}(X)}{}{\Omega\Sigma X} for 
this retraction induces the natural inclusion 
\[\rhlgy{*}{\Omega\Sigma L_{\ell+1}(X)}\simeq T(L_{\ell+1}(V))\cong\ctimes{i=1}{\infty}{S(L_{i}(L_{\ell+1}(V)))}
\subseteq \Lambda(V)\otimes S([L(V),L(V)])\]
into the right-hand factor (where $L_{j}(V)$ denotes the $\zmodp$-submodule of length $i$ Lie brackets in $L(V)$, and 
where the isomorphism follows by the Poincare\text{\'{e}}-Birkhoff-Witt theorem).
Therefore $\Omega\Sigma L_{\ell+1}(X)$ is also a homotopy retract of 
$\Omega Q(X)$. In turn, $\Omega\Sigma^{M+1}X$ is a homotopy retract of $\Omega\Sigma L_{\ell+1}(X)$ using 
Proposition~(\ref{p4}), so we obtain a decomposition
\[\Omega\Sigma X\simeq A(X)\times\Omega\Sigma^{M+1}X\times(\mbox{Some other space}).\]

Since $\ell=\dim V$ is even and $V_{even}=0$, $\rhlgy{*}{\Sigma^{M}X}\cong\Sigma^{M}V$
has only odd degree generators, so we can reapply Cohen's and Neisendorfer's decomposition
to $\Omega\Sigma^{M+1}X$. Iterating this argument, starting by taking $\Sigma^{M}X$ in place of $X$, and 
using an induction similar to the proof of Proposition~(\ref{p4}), we obtain the decomposition 
\[\Omega\Sigma X\simeq \cprod{i=0}{\infty}A(\Sigma^{b_{i,1}} X)\times(\mbox{Some other space}),\]
where $b_{i,1}$ are the integers defined in Proposition~(\ref{p4}).

\end{proof}

\section{An Application to the Moore conjecture}

The $p$-exponent $\exp_{p}(X)$ of a space $X$ is defined as the smallest power $p^t$ that annihilates the $p$-primary torsion of $\pi_{i}(X)$ for all $i>0$. 
Spheres, finite $H$-spaces, and mod-$p$ Moore spaces are all examples of spaces that have finite $p$-exponents at odd primes $p$~\cite{CMN2,Stanley,N2}. 
In the other direction, a simply connected wedge $S^{m}\vee\Sigma X$ with $\Sigma X$ rationally non-trivial does not have a finite $p$-exponent~\cite{Stanley}. 
These isolated examples aside, 
there is no known general set of criteria for distinguishing spaces that have a finite $p$-exponent from those that do not.
However, a conjecture of Moore suggests that the answer is very simple for finite simply connected $CW$-complexes $X$: 
$\exp_{p}(X)$ is finite at any prime $p$ if and only if $\pi_{*}(X)\otimes\mathbb{Q}$ is a finite dimensional vector space. 
When $X$ is a suspension it is known that $X$ is rational wedge of spheres, 
and so in this case the Moore conjecture says that $\exp_{p}(X)$ is finite if and only if $\dim(\hlgy{*}{X;\mathbb{Q}})\leq 1$,

McGibbon and Wilkerson~\cite{MW} were able to give a partial result in one direction of the Moore conjecture.
They showed that a finite simply connect $CW$-complex has a finite $p$-exponent at sufficiently large primes $p$ when it
has finite dimensional rational homotopy. This prime $p$ depends on the given space. 
In the other direction Selick~\cite{Selick5} showed that a finite simply connected $CW$-complex $X$ has no finite $p$-exponent 
whenever $X$ is a suspension, and \hlgy{*}{X;\mathbb{Z}} is torsion-free of rank greater than one. 
This result was in some sense extended by Stelzer~\cite{Stelzer1,Stelzer2} to any finite simply connected $CW$-complex $X$, 
as long as one selects a sufficiently large prime $p$ that depends on the dimension and connectivity of $X$. 
By combining Stelzer's result with that of McGibbon's and Wilkerson's, 
one sees that the Moore conjecture holds in the sense of sufficiently large primes.

There also happens to be a stable analogue of the Moore conjecture due to Stanley~\cite{Stanley},
which has the fortune of being much easier to prove. 
\begin{theorem}[Stanley]
A finite simply connected $CW$-complex $X$ has a finite $p$-exponent on stable homotopy 
$\pi^{s}_{*}(X)$ if and only if $X$ is rationally trivial.
\end{theorem}
Now combining the following proposition with Stanley's theorem, 
we can recover Selick's work on the Moore conjecture when restricted to the spaces in Theorem~(\ref{T0}).
\begin{proposition}
\label{p5}
Take the integers $b_{i}$ and a suspended $p$-local $CW$-complex $X$ as in Theorem~(\ref{T0}), letting 
$V=\rhlgy{*}{X}$, $1<\dim V<p-1$, and either $V_{odd}=0$ or $V_{even}=0$.
Assume $X$ is $(m-1)$-connected for some $m\geq 1$. Then for each $j$ the stable homotopy group
$\pi_{j}^{s}(\Sigma X)$ is a homotopy retract of $\pi_{j+b_{i}}(\Sigma X)$ for all $i$ large enough 
such that $j\leq b_{i}+2m$.
\end{proposition}
\begin{proof}
By Theorem~(\ref{T0}) $\Omega\Sigma^{b_{i}+1}X$ is a homotopy retract of $\Omega\Sigma X$,
so $\pi_{j+b_{i}}(\Sigma^{b_{i}+1} X)$ is a homotopy retract of
$\pi_{j+b_{i}}(\Sigma X)$ for each $j$. By the Freudenthal suspension theorem
$\pi_{j+b_{i}}(\Sigma^{b_{i}+1} X)\cong\pi_{j+b_{i}}^{s}(\Sigma^{b_{i}+1} X)$ 
for $j\leq b_{i}+2m$, and $\pi_{j+b_{i}}^{s}(\Sigma^{b_{i}+1} X)\cong\pi_{j}^{s}(\Sigma X)$.
Thus $\pi_{j}^{s}(\Sigma X)$ is a homotopy retract of $\pi_{j+b_{i}}(\Sigma X)$
when $j\leq b_{i}+2m$.
\end{proof} 
Thus we see that the stable homotopy groups $\pi^{s}_{*}(\Sigma X)$ of the space $\Sigma X$ 
in Proposition~(\ref{p5}) are retracts of $\pi_{*}(\Sigma X)$. 
Since $\Sigma X$ is rationally nontrivial, $exp_{p}(\pi^{s}_{*}(\Sigma X))$ is infinite by Stanley's
theorem. So $exp_{p}(\Sigma X)$ must also be infinite.

One would hope for some sort of generalization of Theorem~(\ref{T0}), beyond the restrictions $V_{odd}=0$ or $V_{even}=0$ .
There are unfortunately many examples where this is impossible.
Let $X$ be a wedge $S^{m}\vee P^{n}(p^{r})$, where the mod-$p$ \emph{Moore space}
$P^{n}(p^{r})$ is the cofibre of the degree $p^{r}$ map \seqm{S^{n-1}}{\underline{p^{r}}}{S^{n-1}}.
Then $\Sigma X$ has torsion in its integral homology, but is rationally nontrivial, 
and the mod-$p$ homology $V=\hlgy{*}{X}$ satisfies $V_{odd}\neq 0$ and $V_{even}\neq 0$. 
If Theorem~(\ref{T0}) applied to this space $X$, 
the $p$-exponent of the stable homotopy groups of $\pi^{s}_{*}(\Sigma X)$
would be bounded above by the $p$-exponent of $\pi_{*}(\Sigma X)$,
and so using Stanley's theorem, $exp_{p}(\Sigma X)$ would be infinite. 
But by application of the Hilton-Milnor theorem to $\Omega\Sigma X$, 
and the fact that $exp_{p}(P^{j}(p^{r}))=p^{r+1}$ independently of $j$~\cite{N2},
$exp_{p}(\Sigma X)$ is in fact finite, a contradiction.

\bibliographystyle{amsplain}
\bibliography{mybibliography}
\end{document}